\documentclass{amsart}
\usepackage[utf8]{inputenc}
\usepackage[T1]{fontenc}
\usepackage{pgf, hyperref, amssymb,enumerate,amsmath}
\usepackage[all]{xy}
\usepackage[capitalize]{cleveref}
\usepackage{mathtools}
\usepackage[shortlabels]{enumitem}
\usepackage{stmaryrd}

\newcommand{\comp}[3]{P_{#1 | #2}(#3)}
\newcommand{\compi}[4]{P^{(#4)}_{#1 | #2}(#3)}
\newtheorem{theo}{Theorem}[section]
\newtheorem*{theo*}{Theorem}
\newtheorem{prop}[theo]{Proposition}
\newtheorem{lem}[theo]{Lemma}
\newtheorem{cor}[theo]{Corollary}

\newtheorem{Rem}[theo]{Remark}

\newtheorem{defi}[theo]{Definition}

\newtheorem{ques}[theo]{Question}

\usepackage{euscript,epsfig}
\usepackage{tikz}
\usetikzlibrary{graphs}
\usetikzlibrary[graphs]
\usetikzlibrary{arrows}
\usepackage{tikz-cd}
\usetikzlibrary{decorations.markings}
\usepackage[colorinlistoftodos]{todonotes}

\def\R{{\mathbb R}}
\def\Z{{\mathbb Z}}

\def\N{{\mathbb N}}

\def\cD{{\mathcal D}}

\def\cB{{\mathcal B}}
\def\cC{{\mathcal C}}
\def\cA{{\mathcal A}}
\def\cL{{\mathcal L}}
\def\cT{{\mathcal T}}

\def\cR{{\mathcal R}}
\def\cS{{\mathcal S}}
\def\cT{{\mathcal T}}
\def\cW{{\mathcal W}}

\def \rcomp {\hbox{\rm r-comp}}

\newcommand{\source}{{\mathsf{s}}}
\newcommand{\range}{{\mathsf{r}}}

\author{Sebasti\'an Donoso}
\address{Departamento de Ingenier\'{\i}a Matem\'atica and Centro de Modelamiento Matem\'atico, Universidad de Chile \& UMI-CNRS 2807, Beauchef 851, Santiago, Chile.} \email{sdonoso@dim.uchile.cl}

\author{Fabien Durand}
\address{Laboratoire Ami\'enois de Math\'ematiques Fondamentales et Appliqu\'ees, CNRS-UMR 7352, Universit\'{e} de Picardie Jules Verne, 33 rue Saint Leu, 80039   Amiens cedex 1, France.} \email{fabien.durand@u-picardie.fr}

\author{Alejandro Maass}
\address{Departamento de Ingenier\'{\i}a Matem\'atica and Centro de Modelamiento Matem\'atico, Universidad de Chile \& UMI-CNRS 2807, Beauchef 851, Santiago, Chile.}\email{amaass@dim.uchile.cl}

\author{Samuel Petite}
\address{Laboratoire Ami\'enois de Math\'ematiques Fondamentales et Appliqu\'ees, CNRS-UMR 7352, Universit\'{e} de Picardie Jules Verne, 33 rue Saint Leu, 80039   Amiens cedex 1, France.} \email{samuel.petite@u-picardie.fr}

\begin{document}
\title{Interplay between finite topological rank minimal Cantor systems, $\mathcal S$-adic subshifts and their complexity}




\thanks{This research was partially supported by grants PIA-ANID AFB 170001 \& Fondap 15090007,  
and grant ANID/MEC/80180045 hosted by University of O'Higgins.
The first and the third author thank the hospitality of the LAMFA UMR 7352 CNRS-UPJV and the ”poste rouge” CNRS program}

\date{\today}

\maketitle

\begin{abstract}

Minimal Cantor systems  of finite topological rank   (that  can be represented by a Bratteli-Vershik diagram with a uniformly bounded number of vertices per level) are known to have dynamical rigidity properties. 
We establish that such systems, when they are expansive, define the same class of systems, up to topological conjugacy, as primitive and recognizable ${\mathcal S}$-adic subshifts.
This is done establishing  necessary and sufficient conditions for a minimal subshift to be of finite topological rank. As an application, we show that minimal subshifts with non-superlinear complexity (like all  classical zero entropy  examples) have finite topological rank. 
Conversely, we analyze the complexity of ${\mathcal S}$-adic subshifts 
and provide sufficient conditions for a finite topological rank subshift to have a non-superlinear complexity. 
This includes minimal Cantor systems given by Bratteli-Vershik representations whose tower levels have proportional heights and the so called left to right ${\mathcal S}$-adic subshifts. We also exhibit that finite topological rank does not imply non-superlinear complexity. 
In the particular case of topological rank 2 subshifts,  we prove their complexity is always subquadratic along a subsequence and their automorphism group is trivial.
\end{abstract}

\markboth{Sebasti\'an Donoso, Fabien Durand, Alejandro Maass, Samuel Petite}{Finite topological rank minimal Cantor systems versus $\cS$-adic subshifts}

\section{Introduction}
\label{sec:Introduction}

Using tools originally appearing in the theory of operator algebras, Herman {\it et al.} \cite{hps} introduced a description of any minimal Cantor system through  a representation of the space and the homeomorphism  by a Bratteli diagram and a Vershik map on it. 
This method proved to be  powerful to solve  ergodic, dynamical and  operator algebras problems, like the computation  of a complete invariant for topological orbit equivalence   \cite{Giordano&Putnam&Skau:1995}, which is the one of the most remarkable achievement.      

A way to describe the simplest minimal Cantor systems is provided by the notion of {finite topological rank}. Analogously to the eponymous notion for measurable dynamics \cite{Ornstein&Rudolph&Weiss:1982}, these  minimal Cantor systems  can be represented by  Kakutani-Rohlin partitions defining a  Bratteli-Vershik diagram with a uniformly bounded number of vertices per level.  
All the classical examples of zero-entropy  minimal Cantor systems (odometers, substitutions subshifts, linearly recurrent subshifts  or symbolic codings of interval exchange maps) have a finite topological rank.  
The  study of finite topological rank systems from decades exhibit rigidity  properties. Let us mention some of them. 
There is a dichotomy in the dynamic: they are  expansive or equicontinuous \cite{Downarowicz&Maass:2008}.
Hence non-equicontinuous finite topological rank systems are conjugate to subshifts and it is a folklore result that they have zero-entropy. 
Also the topological rank is an upper bound for the number of ergodic invariant probability measures  (see \cite{Bezuglyi&Kwiatkowski&Medynets&Solomyak:2013}), for the rational rank of the dimension group \cite{Giordano&Putnam&Skau:1995,Giordano&Handelman&Hosseini:2018} and for the rational rank of the continuous spectrum of the system \cite{Bressaud&Durand&Maass:2010}. 
To understand the (typical) properties of these systems, it is natural to try to  characterize their dynamical properties   through their combinatorial structure given by the Bratteli-Vershik diagrams.  
Their  combinatorial structures  have been exploited most satisfactorily  
in the characterization of Toeplitz subshifts  \cite{gjtoeplitz}, linearly recurrent subshifts \cite{Cortez&Durand&Host&Maass:2003} or  unique ergodicity \cite{Bezuglyi&Kwiatkowski&Medynets&Solomyak:2013}.  
It has been valuable too in the study of the continuous or measurable spectrum of minimal Cantor systems (see \cite{Bressaud&Durand&Maass:2010,Durand&Frank&Maass:2015,Durand&Frank&Maass:2019}) and in the study of the property of orbit equivalence between minimal Cantor systems, many of them being  of finite topological rank (see \cite{Giordano&Putnam&Skau:1995,Cortez&Durand&Petite:2016,Giordano&Handelman&Hosseini:2018, Berthe&Cecchi&Durand&Perrin&Petite:2019}).

Another effective way to represent a minimal subshift is through a $\cS$-adic subshift, that  is a subshift generated by the iteration of a succession of  morphisms (substitutions) on eventually different alphabets, see Section \ref{sec:Sadicsubshifts} for definitions. When  the alphabets are bounded in size,   the $\cS$-adic subshift is said to have  bounded alphabet rank. 
The constructive  description of such subshift leads  to consider  the same problematic on the description of its dynamical properties through the combinatorial data provided by the morphisms.

Few classes of  subshifts can be characterized in combinatorial terms, as substitution subshifts \cite{Durand:1998} and more generally linearly recurrent subshifts \cite{du1}. 
As a main issue in this paper, we want to state sufficient and/or necessary conditions for a minimal subshift to be of finite topological rank or conjugate to a bounded alphabet rank $\cS$-adic subshift.

One of our principal results states under natural conditions a topological equivalence between minimal $\cS$-adic subshifts with bounded alphabet rank and finite topological rank minimal Cantor system (\cref{theo:equivalenceSadicvsFiniteRank}). Positive evidence of this fact was shown in \cite{Berthe&Steiner&Thuswaldner&Yassawi:2018}, where this result was proved  in the measure theoretical category. 
To show it, we  provide a necessary and sufficient condition for a minimal subshift to have a finite topological rank (\cref{theo:FRankSuff}).
The essential key notion for the proofs  is the  recognizability of morphisms together with  the remarkable results in \cite{Berthe&Steiner&Thuswaldner&Yassawi:2018} for $\cS$-adic subshifts extending the seminal recognizability results  for substitutions by B. Mossé in \cite{Mosse:1992} and the related ones in \cite{Bezuglyi&Kwiatkowski&Medynets:2008, Downarowicz&Maass:2008}.
As another application of this criterion, we get that any minimal subshift with non-superlinear word complexity function has finite topological rank (\cref{cor:NonSuperLinareFR}). Hence  it is conjugate to a $\cS$-adic subshift with bounded alphabet rank and has all the aforementioned  dynamical restrictions.
This also recovers, in the minimal setting,  sufficient conditions of \cite{Boshernitzan:1984,Cyr&Kra:2019measures, Ferenczi&Monteil:2010} on the complexity to ensure a finite number of ergodic invariant  probability measures.
Analogous results were obtained in relation to the $\cS$-adic conjecture,  in the measurable setting \cite{Ferenczi:1996} and  in \cite{Durand&Leroy:2012} for minimal subshifts whose complexity difference is at most two.
Along the way and following an idea due to T. Monteil, we  provide a characterization of   Rauzy graphs associated to  minimal subshifts of non-superlinear complexity (\cref{prop:Monteil}).  
Notice nevertheless, that \cref{cor:NonSuperLinareFR} does not admit a converse since we explicitly present  a minimal subshift of topological rank two with a superlinear complexity in Section \ref{sec:Rk2NonSperline}. 
We go deeper into  the relationship between the finite topological rank property and the complexity  of words by studying  the complexity  function within the class of $\cS$-adic subshifts. Upper bounds for it are given in terms of a notion of relative complexity of morphisms  in Section \ref{sec:ComplexityFR}. 
We deduce several applications, among them,  conditions so that a finite topological rank subshift has a non-superlinear complexity. 
Concerning the complexity  of a topological rank two minimal subshift, or of a minimal $\cS$-adic subshift on $2$ letters alphabets, we show it is always sub-quadratic along a subsequence (\cref{cor:SadicNonSuperQuadra}). 
Concerning the orbit equivalence theory, we obtain that any finite topological rank  minimal Cantor system is strongly orbit equivalent to a subshift of sub-linear complexity (\cref{cor:SOEsubLinComp}).   
We also exhibit, thanks to recognizability results,  a rigid dynamical property of topological rank two systems: they have only one asymptotic component (\cref{cor:ACRank2}). It follows, by using arguments developed in \cite{Donoso&Durand&Maass&Petite:2016}, that their automorphism groups are always trivial, {\it i.e.} generated by powers of the shift map.

\subsection{Organization} 
In the next section we give the basic background in topological and symbolic dynamics needed in this article. \cref{sec:recog-factorizing} is devoted to review and study recognizability properties of a sequence of morphims. The equivalence between finite topological rank and $\cS$-adic subshifts is established in 
\cref{section:finiterankvsSadic} and in \cref{sec:StrongDec} we show that non-superlinear subshifts have a finite topological rank. 
We investigate the complexity of $\cS$-adic subshifts in \cref{sec:ComplexityFR} and provide conditions to have a non-superlinear/linear complexity. 
In the last but one section  we analyze topological properties of topological rank two subshifts, focusing on the number of asymptotic components. The definitions needed are included at the beginning of this section, so that the interested reader can read it independently of the others. 
We conclude with some remarks and questions in the last section.

In this article, we let $\N, \N^{\ast}$ and $\Z$ denote the set of non-negative integers, the set of positive integers and the set of integers numbers respectively. 


\section{Basics in topological and symbolic dynamics}
\label{sec:Basics}

\subsection{Basics in topological dynamics}\label{subsec:topdyn}
A {\em topological dynamical system} (or just a system) is a pair $(X,T)$ where $X$ is a compact metric space and  $T\colon X \to X$ is a homeomorphism.  We denote by $\text{Orb}_T(x)$ the orbit $\{T^nx : n\in \Z\}$ of $x \in X$. A point $x\in X$ is {\em periodic} if $\text{Orb}_T(x)$ is a finite set and {\em aperiodic} otherwise. A topological dynamical system is {\em aperiodic} if any point $x\in X$ is aperiodic and is {\em minimal} if the orbit of every point is dense in $X$.

A {\em factor map} between the topological dynamical systems $(X,T)$ and $(Y,S)$ is a continuous onto map $\pi\colon X\to Y$ such that $\pi\circ T=S\circ \pi$. We say that $(Y,S)$ is a {\em factor} of $(X,T)$ and that $(X,T)$ is an {\em extension} of $(Y,S)$. We use the notation $\pi\colon (X,T)\to (Y,S)$ to indicate the factor map. If in addition $\pi$ is a bijective map we say that $(X,T)$ and $(Y,S)$ are {\em topologically conjugate} and we call $\pi$ a topological conjugacy. 

\subsection{Basics in symbolic dynamics}\label{subsec:symbdyn}

\subsubsection{Subshifts and their languages} 

Let ${\mathcal A}$ be a finite set that we call {\em alphabet}. Elements in ${\mathcal A}$ are called {\em letters} or {\em symbols}. The set of finite sequences or {\em words} of length $\ell\in \N$ with letters in $\mathcal A$ is denoted by ${\mathcal A}^\ell$, the set of onesided sequences $(x_n)_{n\in \mathbb{N}}$  in ${\mathcal A}$ is denoted  by ${\mathcal A}^{\mathbb N}$
and the set of twosided sequences $(x_n)_{n\in \mathbb{Z}}$  in ${\mathcal A}$ is denoted by ${\mathcal A}^{\mathbb Z}$. Also, a word $w= w_1 \cdots w_{\ell} \in  {\mathcal A}^\ell$ can be seen as an element of the free monoid ${\mathcal A}^*$ endowed with the operation of concatenation. 
The integer $\ell$ is the {\em length} of $w$  and is denoted by $|w|=\ell$.
For two finite words $u$ and $v$, we say they are respectively a {\em prefix} and a {\em suffix} of the word $uv$.

The {\em shift map} $S \colon {\mathcal A}^{\mathbb Z} \to {\mathcal A}^{\mathbb Z}$ is defined by $S ((x_n)_{n\in \mathbb{Z}}) = (x_{n+1})_{n\in \mathbb{Z}}$. To simplify notations we denote the shift map by $S$ independently of the alphabet, this will be clear from the context. A {\em subshift} is a topological dynamical system $(X,S)$ where $X$ is a closed and $S$-invariant subset of ${\mathcal A}^{\mathbb Z}$ (we consider the product topology in ${\mathcal A}^{\mathbb Z}$). Classically one identifies $(X,S)$ with $X$, so one says that $X$ itself is a subshift. When we say that a sequence $x$ in a subshift is aperiodic, we implicitly mean that $x$ is aperiodic for the action of the shift.  
For convenience, when we state general results about topological dynamical systems we use the notation $(X,T)$ and to state specific results about subshifts we use $(X,S)$. 

Let $(X,S)$ be a subshift. The {\em language} of $(X,S)$ is the set 
${\mathcal L}(X)$ containing all words $w \in {\mathcal A}^*$ such that 
$w=x_m \cdots x_{m+|w|-1} = x_{[m,m+|w|)}$ for some $x=(x_n)_{n\in\Z} \in X$ and $m\in \Z$. 
In this case we say that $w$ {\em appears} or {\em occurs} in the sequence $x$ and that the index $m$ is an {\em occurrence} of $w$ in $x$. 
We also say that $w$ {\em starts} (resp. {\em ends}) in a word $u$ whenever $u = x_{[m-i, m+j)}$ (resp. $u = x_{[m+|w|-1-i, m+|w|-1+j)}$) for some indices $i \ge 0$, $j >0$. We use  the same name for the similar notions for finite words $x$. 
We denote by ${\mathcal L}_\ell (X)$ the set of words of length $\ell$ in ${\mathcal L} (X)$ and by ${\mathcal L}_{\geq\ell} (X)$ the set of words of length greater than or equal to $\ell$ in ${\mathcal L}(X)$. Given $x \in X$ the language ${\mathcal L}(x)$ is the set of all words appearing in $x$. As before we define ${\mathcal L}_{\ell} (x)$ and ${\mathcal L}_{\geq\ell} (x)$.
For two words $u,v \in {\mathcal L} (X)$, the {\em cylinder set} $[u.v]$ is the set $\{x \in X: x_{[-|u|,|v|)} = uv\}$. When $u$ is the empty word we only write $[v]$, erasing the dot. Cylinder sets are {\em clopen sets} (closed and open sets). They form a base for the topology of the subshift. 

A word $w\in {\mathcal L} (X)$ is said to be {\em right special} if there exist at least two distinct letters $a$ and $b$ such that $wa$ and $wb$ belong to ${\mathcal L}(X)$. In the same way we define {\em left special} words. We denote by $\mathcal{LS}(X)$ and $\mathcal{RS}(X)$ the set of left and right special words in $X$ respectively. It is well known that for any infinite subshift $X$ admits left and right special words of any length.

The non decreasing map $p_X \colon {\mathbb N} \to {\mathbb N}$ defined by $p_X (n) = |{\mathcal L}_n (X)|$ is called the {\em complexity function} of $X$. If  $\limsup_{n\to +\infty} {p_X(n)}/{n} < +\infty$ we say that $X$ has {\em sub-linear} complexity. Equivalently, if $p_X(n)\leq cn$ for all $n\in \N$ for some positive constant $c$. 
We say that $X$ has {\em non-superlinear} complexity if $\liminf_{n\to +\infty} {p_X(n)}/{n} < +\infty$. In the contrary, we say $X$ has {\em superlinear} complexity.   

\subsubsection{Morphisms and substitutions}

For an alphabet $\cA$, let $\cA^{\ast}$ denote the free monoid generated by concatenations of letters in $\cA$. 
The neutral element is called the \emph{empty word}.
Let $\cA$ and $\cB$ be finite alphabets and $\tau \colon \cA^{\ast} \to \cB^*$ be a morphism.  
We say that $\tau$ is {\em erasing} whenever there exists $a\in \cA$ such that $\tau (a)$ is the empty word. Otherwise we say it is {\em non-erasing}. 
The morphism $\tau$ is {\em proper}  
if each word $\tau(a)$ starts and ends with the same letter independently of $a$. 
When it is non-erasing it extends naturally to maps from $\cA^\mathbb{N}$ to itself and  from $\cA^\mathbb{Z}$ to itself  in the obvious way by concatenation (in the case of a twosided sequence we apply $\tau$ to positive and negative coordinates separately and we concatenate the results at the coordinate zero). To the morphism $\tau$ we associate an {\em incidence matrix} $M_\tau$ indexed by $\cA\times \cB$ such that its entry at position $(a,b)$ is the number of occurrences of $b$ in $\tau(a)$ for every $a \in \cA$ and $b \in \cB$. When all the entries of $M_{\tau}$ are positive, we say that $\tau$ is {\em positive}.
Note that any function $\tau\colon \cA\to \cB^{\ast}$ can be naturally extended  
to a morphism (that we also denote $\tau$) from $\cA^{\ast}$ to $\cB^{\ast}$ by concatenation, and we use this convention throughout the document. So from now on, all functions between finite alphabets are considered to be morphisms between their free monoids.

\subsubsection{$\cS$-adic subshifts}\label{sec:Sadicsubshifts}

We recall the definition of a $\cS$-adic subshift as stated in \cite{Berthe&Steiner&Thuswaldner&Yassawi:2018}.
A {\em directive sequence}  $\boldsymbol{\tau} = (\tau_n \colon \mathcal{A}_{n+1}^{\ast}\to \mathcal{A}_n^*)_{n \geq 0}$ is a sequence of morphisms. From now on, we only consider morphisms that are non-erasing. When all morphisms $\tau_n$, $n\geq 1$, are proper we say that $\boldsymbol{\tau}$ is {\em proper} and when 
all incidence matrices of $\tau_n$, $n\geq 1$, are positive, we say that $\boldsymbol{\tau}$ is {\em positive}. Here we stress the fact that in the definition there is no assumption (more that non-erasingness) on the first morphism $\tau_0$, since in many cases of interest this morphism is not proper nor positive. 

For $0\leq n<N$, we denote  by  $\tau_{[n,N)}$ or $\tau_{[n,N-1]}$, the morphism $\tau_n \circ \tau_{n+1} \circ \dots \circ \tau_{N-1}$.
We say $\boldsymbol{\tau}$ is {\em primitive} if for any $n\in \N$ there exists $N>n$ such that 
$M_{\tau_{[n,N)}} >0$, {\em i.e.}, $\tau_{[n,N)}(a)$ contains all letters in $\cA_{n}$ for all $a\in \cA_{N}$.

For $n\in \N$, the \emph{language $\mathcal{L}^{(n)}({\boldsymbol{\tau}})$ of level $n$ associated 
with~$\boldsymbol{\tau}$} is defined by 
$$
\mathcal{L}^{(n)}({\boldsymbol{\tau}}) = \big\{w \in \mathcal{A}_n^* : \ \mbox{$w$ occurs in $\tau_{[n,N)}(a)$ for some $a \in\mathcal{A}_N$ and $N>n$}\big\}
$$
and $X_{\boldsymbol{\tau}}^{(n)}$ is the set of points $x \in \mathcal{A}_n^\mathbb{Z}$ such that $\cL (x) \subseteq \mathcal{L}^{(n)}({\boldsymbol{\tau}})$. This set clearly defines a subshift 
that we call the {\em subshift generated by~$\mathcal{L}^{(n)}({\boldsymbol{\tau}})$}. It may happen that 
$\cL(X_{\boldsymbol{\tau}}^{(n)})$ is strictly contained in $\mathcal{L}^{(n)}({\boldsymbol{\tau}})$, but in the primitive case both sets coincide and $X_{\boldsymbol{\tau}}^{(n)}$ is a minimal subshift. 
By definition we have $\tau_n(\cL(X_{\boldsymbol{\tau}}^{(n+1)})) = \cL(X_{\boldsymbol{\tau}}^{(n)})$.
We set $X_{\boldsymbol{\tau}} = X_{\boldsymbol{\tau}}^{(0)}$ and call $(X_{\boldsymbol{\tau}},S)$ the \emph{$\cS$-adic subshift} generated by the {directive sequence}~$\boldsymbol{\tau}$.  

We define the {\em alphabet rank} of $\boldsymbol{\tau}$ as 
$$AR(\boldsymbol{\tau}) = \liminf_{n\to +\infty} |  \mathcal{A}_n|.$$
When $AR (\boldsymbol{\tau})$ is finite this means $(X_{\boldsymbol{\tau}},S)$ can be defined by a directive sequence where the morphisms are endomorphisms on a free monoid with  $AR(\boldsymbol{\tau})$ generators. 

We define the  {\em alphabet rank} of a subshift  $X$ (or $(X ,S)$) to be 
$$
AR (X) = \inf_{\boldsymbol{\tau},X_{\boldsymbol{\tau}} = X} AR (\boldsymbol{\tau}),
$$
where the infimum runs over all sequences of morphisms $\boldsymbol{\tau}$ such that $X_{\boldsymbol{\tau}} = X$. As a convention the infimum is equal to 
 $+\infty$ when this set is empty.

\subsection{Bratteli-Vershik systems}\label{subsec:Bratteli-Vershik}

Let $(X,T)$ be a minimal Cantor system, {\em i.e.}, $X$ is a Cantor set (has a countable basis of closed and open sets and it has no isolated points). It can be represented by an ordered Bratteli diagram together with the Vershik transformation acting on it. We give a brief outline of this construction, for details on this theory see \cite{hps} or \cite{review}.

\subsubsection{Bratteli diagrams}\label{subsec:BratteliDiagrams}
A Bratteli diagram is an infinite graph which consists of a vertex set $V$ 
and an edge set $E$, both of
which are divided into levels $V=V_{0}\cup V_{1}\cup \ldots$ and $E=E_{1}\cup E_{2}\cup \ldots$, 
where all levels are pairwise disjoint. 
The set $V_{0}$ is a singleton $\{v_{0}\}$ and for all $n\geq 1$ edges in $E_{n}$ join vertices in $V_{n-1}$ to vertices in $V_{n}$. If $e\in E$ connects $u\in V_{n-1}$ with $v \in V_n$ we write $\source(e)=u$ and $\range(e)=v$, where $\source\colon E_n\to V_{n-1}$ and $\range\colon E_n\to V_{n}$ are the source and range maps, respectively. It is also required that $\source^{-1}(v)\not = \emptyset$ for all $u \in V$ and that $\range^{-1}(v)\not = \emptyset$ for all $v \in V\setminus V_0$.
For all $n\geq 1$ we set $|V_{n}|=d_{n}$ and we write $V_n=\{1,\ldots,d_n\}$ to simplify notation.   

Fix $n\geq 1$. We call \emph{level} $n$ of the diagram the subgraph consisting of the vertices in $V_{n-1}\cup V_{n}$ 
and the edges $E_{n}$ between these vertices. Level $1$ is called the \emph{hat} of the Bratteli diagram. 
We describe the edge set $E_n$ using a $V_{n}\times V_{n-1}$ incidence matrix $M_{n}$ for which its $(v,u)$ entry is the number of edges in $E_{n}$ joining vertex $u \in V_{n-1}$ with vertex $v \in V_{n}$. 
We also consider several levels at the same time. 
For integers $0 \leq n< N$ we denote by $E_{n,N}$ the set of all paths in the graph joining vertices of  
$V_{n}$ with vertices of $V_{N}$. 

A Bratteli diagram $(V,E)$ is \emph{simple} or {\em primitive} if for any $n\geq 1$ there exists $N > n$ such that each pair of vertices $u\in V_n$ and $v \in V_N$ is connected by a finite path.

\subsubsection{Ordered Bratteli diagrams and Bratteli-Vershik representations}\label{sec:BVrepresentation}

An \emph{ordered} Bratteli diagram is a triple \( B=\left( V,E,\preceq \right) \), where \( \left( V,E\right)  \) is a Bratteli diagram and \( \preceq  \) is a partial ordering on \( E \) such that: edges
\( e \) and \( e' \) in $E$ are comparable if and only if $\range(e)=\range(e')$.
This partial ordering naturally defines maximal and minimal edges. The partial ordering of $E$ induces another one on paths of $E_{n,N}$ for all $0 \leq n < N$:  
$\left(e_{n+1},\ldots,e_{N}\right) \preceq \left(f_{n+1},\ldots ,f_{N}\right)$ if and only if 
there is $n+1\leq i\leq N$ such that $e_{i}\preceq f_{i}$ and $e_{j}=f_{j}$ for $i<j\leq N$.

Given a strictly increasing sequence of integers 
$\left(n_{k}\right)_{k\geq 0}$ with $n_{0}=0$ one defines the \emph{contraction} or \emph{telescoping} of
$B=\left(V,E,\preceq \right)$ with respect to $\left(n_{k} \right)_{k\geq 0}$ by 
$$\left(\left(V_{n_{k}}\right)_{k\geq
0},\left( E_{n_{k},n_{k+1}}\right)_{k\geq 0},\preceq \right), $$ where $\preceq$ is the order induced in each set of edges 
$E_{n_{k},n_{k+1}}$. The converse operation is called {\em microscoping} (see \cite{hps} and \cite{Giordano&Putnam&Skau:1995} for more details).

Given an ordered Bratteli diagram \( B=\left( V,E,\preceq \right) \) one defines \( X_{B} \) as the set of infinite paths \( \left(
x_{1},x_{2},\ldots \right)  \) starting in \( v_{0} \) such that $\range(x_n)=\source(x_{n+1})$
for all \( n\geq 1 \). We topologize \( X_{B} \) by postulating a basis of open sets, namely the family of \emph{cylinder sets}
$$
\left[ e_{1},e_{2},\ldots ,e_{n}\right]
=\left\{ \left( x_{1},x_{2},\ldots \right) \in X_{B} \textrm{ } : \textrm{ }
x_{i}=e_{i},\textrm{ for }1\leq i\leq n\textrm{ }
\right\}.$$

Each \( \left[ e_{1},e_{2},\ldots ,e_{n}\right]  \) is also closed, as is
easily seen, and so \( X_{B} \) is a compact, totally disconnected metrizable space. 
If $(V,E)$ is simple and $X_B$ infinite then \( X_{B} \) is Cantor.

When there is a unique point  \( \left( x_{1},x_{2},\ldots \right) \in X_{B} \) such that \( x_{n} \) is (locally) maximal for any \( n\geq 1 \) and a unique
point \( \left( y_{1},y_{2},\ldots \right) \in X_{B} \) such that \( y_{n} \) is (locally) minimal for any \(n \geq 1 \), one says that \( B=\left(
V,E,\preceq \right)  \) is a \emph{properly ordered} Bratteli diagram. We call these particular points \( x_{\mathrm{max}} \) and \(
x_{\mathrm{min}} \) respectively. In this case, we define the map \( V_{B} \) on \( X_{B} \) called the \emph{Vershik map} as follows. Let \( x=\left( x_{1},x_{2},\ldots \right) \in X_{B}\setminus \left\{ x_{\mathrm{max}}\right\}  \) and let \(
n\geq 1 \) be the smallest integer so that \( x_{n} \) is not a maximal edge. Let \( y_{n} \) be the successor of \( x_{n} \) for the corresponding local order and \( \left(
y_{1},\ldots ,y_{n-1}\right)  \) be the unique minimal path in \( E_{0,n-1} \) connecting \( v_{0} \) with the initial vertex of \( y_{n} \). We
set \( V_{B}\left( x\right) =\left( y_{1},\ldots ,y_{n-1},y_{n},x_{n+1},\ldots \right)  \) and \( V_{B}\left( x_{\mathrm{max}}\right)
=x_{\mathrm{min}} \).

The couple \( \left( X_{B},V_{B}\right)  \) is a topological dynamical system called the \emph{Bratteli-Vershik system} generated by \( B=\left( V,E,\preceq \right)
\). The dynamical system induced by any telescoping of \( B \) is topologically conjugate to \( \left( X_{B},V_{B}\right)  \). 

In \cite{hps} it is proved that the system $\left( X_{B},V_{B}\right)$ is minimal whenever
the associated Bratteli diagram $(V,E)$ is simple. Conversely, it is also proved that any minimal Cantor system $\left( X,T\right)$ is topologically conjugate to a Bratteli-Vershik system \( \left(X_{B},V_{B}\right)  \) where $(V,E)$ is simple. 
We say that a simple and properly ordered Bratteli diagram $B=(V,E,\preceq)$ is a \emph{Bratteli-Vershik representation} of the minimal Cantor system $\left( X,T\right)$ 
if $(X,T)$ and $\left( X_{B},V_{B}\right)$ are topologically conjugate. We also say that $(X,T)$ is given by the Bratteli-Vershik representation $B$. 

Let $B=( V, E ,\preceq )$ be an ordered Bratteli diagram. We associate different subshifts to the dynamics of $(X_{B}, V_{B})$.
First, for each $n \ge 1$ set 
$$X_B^{(n)}=
\{ (\range(\pi_{n}(V_{B}^{k}x)))_{k\in \Z}: x=(x_i)_{i\geq 1} \in X_{B} \} \subseteq V_{n}^{\Z},$$ 
where 
$\pi_{n}: X_B \to E_{n}$ is the projection map given by $\pi_{n}(x)=x_{n}$. We call this subshift the {\em coding of the orbits of} $(X_B,V_B)$ {\em at level} $n$. The subshift $(X_B^{(n)},S)$ is minimal whenever the system $(X_B, V_B)$ is.
Now, let 
$X_B^{(0)}= \{ (\pi_{1}(V_{B}^{k}x))_{k\in \Z} :  x \in X_{B} \} \subseteq E_{1}^{\Z}$. Denote by $\phi_{0}: X_B \to X_B^{(0)}$ the natural factor map between $(X_B,V_B)$ and $(X_B^{(0)},S)$. 
It is clear that the subshifts $X_B^{(0)}$ and $X_B^{(1)}$ coincide, up to the alphabets, whenever there is a unique edge between the vertex in $V_0$ and each vertex in $V_1$. 

Define $\tau_0^B:V_1^* \to E_1^*$ by 
$\tau^B_0(v)=e_1(v)\ldots e_\ell(v)$, where the edges $e_1(v),\ldots,e_\ell(v)$ correspond to those connecting $V_0$ with $v$ and are ordered following the partial order of $B$. 
In a similar way, for $n\geq 1$ consider the morphisms 
$\tau_n^{B} \colon V_{n+1}^*\to V_{n}^*$ defined for a vertex $v\in V_{n+1}$ by $\tau_n^{{B}} (v) = \source(e_1 (v)) \cdots \source(e_l (v))$, where  edges $ e_1 (v), \ldots , e_l (v)$ are all the edges in $E_{n+1}$ whose range is $v$ ordered as indicated by the partial order of $B$. 
We say $\tau_n^B$ is the morphism {\em read at level} $n+1$ on $B$ and that $\boldsymbol{\tau}^B = (\tau_{n}^B)_{n\geq 0}$ is the {\em sequence of morphisms read on} $B$. Following these definitions, an ordered Bratteli diagram $B$ defines a unique sequence of morphisms $(\tau^B_n \colon \cA_{n+1}^* \to \cA_{n}^*)_{n\geq 0}$ with $\cA_0= E_1$ and $\cA_n=V_{n}$ for $n\ge 1$.  
In the sequel it will be useful to observe that for the first morphism $\tau_0^{B}$ all the letters occurring in $\tau^B_0 (v)$, $v\in V_1$, are pairwise distinct. 

The sequence $\boldsymbol{\tau}^B$ is primitive whenever the Bratteli diagram is simple.   
Observe moreover  that $\tau_{n}^{B}(X_B^{(n+1)})$ is included in $X_B^{(n)}$,
and  $\cup_{k\in \mathbb{Z}} S^k \tau_{n}^{B}(X_B^{(n+1)})=   X_B^{(n)}$.
Hence $X_B^{(0)}$ is the $\cS$-adic  subshift generated by $\boldsymbol{\tau}^B$.

It is standard to check that the unicity of the maximal and minimal paths in a properly ordered Bratteli diagram does imply, up to telescoping, that the morphisms $\tau_n^B$ are proper for all $n\geq 1$.  

\subsubsection{Kakutani-Rohlin partitions and Bratteli-Vershik representations}
The construction of a Bratteli-Vershik representation for a minimal Cantor system $(X,T)$ relies on the notion of {\em clopen Kakutani-Rohlin partition (CKR partition)}. A CKR partition $\cT$ of $X$ is given by: 
\begin{equation}
\label{eq:CKR-partition}
\cT =\{T^{-j}B(k): \  1\leq k \leq d , \ 0 \leq j< h(k) \},
\end{equation}
where $d,h(1),\ldots,h(d)$ are positive integers and 
$B(1),\ldots, B(d)$ are clopen subsets of $X$ such that 
$$\bigcup_{k=1}^d T^{-h(k)} B(k)= \bigcup_{k=1}^d B(k).$$
The {\em base} of $\cT$ is the set $B({\cT})=\bigcup_{1\leq k\leq d}B(k)$ and the 
set $\cT(k)=\{T^{-j}B(k): \ 0 \leq j< h(k) \}$ is the {\em $k$-th tower} of $\cT$ and $h(k)$ its {\em height}.
\smallskip

To construct a Bratteli-Vershik representation of $(X,T)$ we need a sequence of {\em nested} CKR partitions.
For each $n\geq 0$ let
\begin{equation}
\label{eq:def-seq-KR}
\cT_n=
\{
T^{-j}B_{n}(k): \ 1\leq k \leq d_n,\ 0\le j <h_{n}(k)
\}
\end{equation}
such that $\cT_0$ is the
trivial partition, ({\em i.e.}, $d_0=1$, $h_0(1)=1$ and $B_0(1)=X$). 
We say that the sequence of CKR partitions $(\cT_n)_{n\geq 0}$ is \emph{nested} if for any $n\geq 0$: 
\begin{enumerate}[label=(KR\arabic*)]
\item
\label{item:KR1}
$B(\cT_{n+1}) \subseteq B(\cT_n)$;
\item
\label{item:KR2}
$\cT_{n+1} \succeq \cT_n$, {\em i.e.}, for any atom $A \in \cT_{n+1}$
there exists an atom  $A^{'}\in \cT_n$ such that $A\subseteq A^{'}$;
\item
\label{item:KR3}
$\bigcap_{n\geq 0} B(\cT_n)$ consists of a unique point;
\item
\label{item:KR4}
the sequence of partitions spans the topology of $X$.
\end{enumerate}
\smallskip
These conditions imply in particular that for all $n\geq 1$ and $1 \leq k \leq d_{n+1}$
all points in $T^{-(h_{n+1}(k)-1)}B_{n+1}(k)$ (the {\em roof} of the $k$-th tower) iterated by $T$ cross completely and in the same order several towers of level $n$ until reaching $B_{n+1}(k)$. Recursively one has that 
for all $1 \leq n < N$ and $1 \leq k \leq d_{N}$ all points in $T^{-(h_{N}(k)-1)}B_{N}(k)$ iterated by $T$ cross completely, and in the same order, several towers of level $n$ until reaching the base $B_{N}(k)$.
Not necessarily all towers of level $n$ are visited. We say that the nested sequence of CKR partitions 
$(\cT_n)_{n\geq 0}$ is {\em primitive} if for all $n\geq 1$ there exists $N>n$ such that for all $1 \leq k \leq d_{N}$ all points in $T^{-(h_{N}(k)-1)}B_{N}(k)$ iterated by $T$ cross completely and in the same order all towers of level $n$ until reaching the base $B_{N}(k)$. 

Given a primitive sequence of nested CKR partitions of $(X,T)$ it is possible to construct a Bratteli-Vershik representation of $(X,T)$ and vice versa. We will briefly describe the construction of the representation and the converse will be clear from this construction too. First, define the associated ordered Bratteli diagram (we follow notations of Section \ref{subsec:BratteliDiagrams}). Fix $n\geq 0$. The vertices $V_n$ will represent the towers in $\cT_n$ so we set $V_n=\{1,\ldots,d_n\}$, in particular $V_0=\{1\}$. In the set of edges $E_{n+1}$ we put an arrow from vertex $u \in V_n$ to vertex $v \in V_{n+1}$ each time the tower $u$ of level $n$ is crossed by points in tower $v$ of level $n+1$ as discussed before. The local order at the vertex $v \in V_{n+1}$ is given by the order points in $B_{n+1}(v)$ cross the towers of level $n$ when iterated by $T^{-1}$. In this way we have defined an ordered Bratteli diagram. By construction and primitivity assumption the constructed diagram is simple. The fact that it is properly ordered follows from \ref{item:KR3}. Indeed, the unique maximal point $x_{\mathrm{max}}$ is given by \ref{item:KR3} and the unique minimal point
$x_{\mathrm{min}}$ is its image by $T$. The conjugacy between $(X_B,V_B)$ and $(X,T)$ is described as follows. A point in $x\in X$ is determined uniquely by the intersection 
$\{x\}=\cap_{n\geq 0} T^{-j_{n,x}}B_n(u_{n,x})$, where $1\leq u_{n,x} \leq d_n,\ 0\le j_{n,x} <h_{n}(u_{n,x})$. Then for each $n\geq 0$ we select the edge $e_{n+1}$ with extremities $u_{n,x} \in V_n$ and $u_{n+1,x} \in V_{n+1}$ whose local order corresponds to the apparition of the $u_{n,x}$-th tower of level $n$ containing 
$T^{-j_{n+1,x}}B_{n+1}(u_{n+1,x})$. The map $x \to (e_1,e_2,\ldots)$ is a conjugacy.  
Thus we have defined a Bratteli-Vershik representation for $(X,T)$. 


\subsubsection{Minimal Cantor systems of finite topological rank}

A minimal Cantor system is of finite topological rank if it admits a Bratteli-Vershik representation such that the number of vertices per level is uniformly bounded by some integer $d$. The minimum possible value of $d$ is called the \emph{topological rank} of the system. 

Equivalently (due to previous section), a minimal Cantor system is of finite topological rank if it admits a primitive sequence of nested CKR partitions such that the number of towers per level is uniformly bounded by some $d\geq 1$. When $d=1$ the system is called an {\em odometer}. 

A fundamental result concerning finite topological rank systems states the following dicothomy.

\begin{theo}{\cite{Downarowicz&Maass:2008}}\label{DowMaass}
A finite topological rank minimal Cantor system is either an odometer or expansive (so a subshift).
\end{theo}
It is important to notice that in the expansive case the Cantor system is topologically conjugate to one of the level subshifts $(X_B^{(n)},S)$ associated to a Bratteli-Vershik representation $B$ of finite topological rank.  

For other proofs of this result and extensions to non minimal contexts see \cite{Bezuglyi&Kwiatkowski&Medynets:2008,Hoynes:2017,Berthe&Steiner&Thuswaldner&Yassawi:2018,Shimomura:2017}.

\section{Recognizability for morphisms and languages}
\label{sec:recog-factorizing}

The notion of recognizability in symbolic dynamics is mainly known because of the result by B. Moss\'e stating that all primitive substitutions or the corresponding subshifts are recognizable \cite{Mosse:1992}. A possible way to state this property is that points in a substitutive subshift can be uniquely factorized using words of a given finite set or the same for families of finite sets of words. In the case of substitutions those families correspond to the words produced by iterating letters. This notion seen in this way can be naturally extended to many other systems, in particular to $\cS$-adic subshifts \cite{Berthe&Steiner&Thuswaldner&Yassawi:2018} or subshifts defined from a Bratteli-Vershik system \cite{Downarowicz&Maass:2008}. It is worth mentioning that Theorem \ref{DowMaass} can be seen as a recognizability result. Now we precise the notions of factorization and recognizability we will need in this article. For definitions of recognizability we mainly follow \cite{Berthe&Steiner&Thuswaldner&Yassawi:2018}. 

\subsection{Factorization of sequences}\label{sec:FactSeq} 
Let $\cA$ be a finite alphabet and $\cW$ be a finite set of non-empty words in $\cA^*$. 
A $\cW$-{\em factorization} of $x\in \cA^\Z$ is a map 
$F_x \colon \Z \to \cW \times \Z$ such that:
\begin{itemize}
\item for each $k \in \Z$, if  $F_x(k)=(w,i)$ and $F_x(k+1)=(w',j)$ then $x_{[i,j)}=w$;
\item if $F_x(0) = (w,i)$ and $F_x(1)=(w',j)$ then $i\leq 0 < j$.
\end{itemize}
For $F_x(k)=(w,i)$ we say $i \in \Z$ is a {\em starting position} of $w$ in $x$ for $F_x$. The starting positions are also called {\em cuts} (see \cite{Berthe&Steiner&Thuswaldner&Yassawi:2018,Downarowicz&Maass:2008}). In other words, $x$ admits a $\cW$-factorization if there exists a strictly increasing sequence of integers $(n_k)_{k\in \Z}$ such that:  $n_{0}\leq 0 <n_1$ and $x_{[n_k,n_{k+1})} \in \cW$ for all $k\in \Z$. Of course, this sequence corresponds to the sequence of starting positions of the factorization: $F_x (k) = (x_{[n_k,n_{k+1})} , n_k)$ for all $k\in \mathbb{Z}$.
Two factorizations $F_x$, $F'_x$ of $x$ are {\em disjoint} whenever the starting positions 
of $F_x$ and $F'_x$ are distinct, so they do not have common cuts. 

We recall a theorem that is very useful when studying factorization of finite or infinite words. It is a consequence of the so called ``defect theorems'' in language theory. We will use it to improve some results concerning recognizability of $\cS$-adic systems in the next section. 

\begin{theo}[{\cite[Corollary~1]{Karhumaki&Manuch:2002}}]
\label{theo:Wfact}
Let $\cW\subset \cA^*$ be a finite set of non-empty words and $x\in \cA^\Z$ be a non-periodic point. 
Then, $x$ possesses at most $|\cW|$ disjoint $\cW$-factorizations.
\end{theo}





It is interesting to mention that in the proof of Theorem \ref{DowMaass} one can replace the so called ``infection lemma'' by Theorem \ref{theo:Wfact}. The original proof of  Theorem \ref{DowMaass} is by contradiction, reducing the study to the case where there are more than $|\cW|^{|\cW|}$ disjoint factorizations of a given point for some finite set of non-empty words $\cW$. Using Theorem \ref{theo:Wfact} the constant $|\cW|^{|\cW|}$ can be reduced to $|\cW|+1$ answering the last question in \cite{Downarowicz&Maass:2008}. We will see in the next section that the related result, Theorem 5.1 from \cite{Berthe&Steiner&Thuswaldner&Yassawi:2018}, can also be improved using the same theorem.

\subsection{Recognizability for morphisms}\label{subsec:recogformorphisms}

Let $\tau\colon \mathcal{A}^* \to \mathcal{B}^*$ be a morphism and $x \in \mathcal{B}^\mathbb{Z}$.
If $x = S^k \tau(y)$ for some $y \in \mathcal{A}^\mathbb{Z}$ and $k \in \mathbb{Z}$, then 
$(k,y)$ is said to be a \emph{$\tau$-representation}  of $x$. 
If $y$ belongs to some $Y\subseteq \mathcal{A}^\mathbb{Z}$ we say $(k,y)$ is a $\tau$-representation of $x$ in $Y$.
Moreover, if $0 \leq k < |\tau(y_0)|$ 
then $(k,y)$ is a \emph{centered $\tau$-representation} of $x$ in $Y$. 
Following the notion introduced in \cite{Berthe&Steiner&Thuswaldner&Yassawi:2018}, we say that $\tau$ is \emph{recognizable in~$Y$} if each 
$x \in \mathcal{B}^\mathbb{Z}$ has at most one centered $\tau$-representation in $Y$. 
If any aperiodic point $x \in \mathcal{B}^\mathbb{Z}$ has at most one centered $\tau$-representation in~$Y$, we say that $\tau$ is \emph{recognizable in~$Y$ for aperiodic points}.

The following lemma is proven in \cite{Berthe&Steiner&Thuswaldner&Yassawi:2018}  using a contradiction argument that does not give any information on the constant $R$ below. We provide a different and direct proof showing the origin of such constant $R$.

\begin{lem}
\label{lem:mosserec}
Let $Y \subseteq \cA^{\Z}$ be a subshift and $\tau \colon \cA^* \to \cB^*$ be a non-erasing morphism. Let $X \subseteq \cB^\Z$ be the smallest subshift containing $\tau (Y)$.
Then, the following are equivalent:
\begin{enumerate}
\item $\tau$ is recognizable in $Y$; 
\item $\tau([a])$ is a non-empty clopen subset of $X$ for each $a\in \cA$, where $[a]=\{y=(y_n)_{n\in \Z} \in Y : y_0=a\}$;
\item \label{item:recradius} there exists a positive integer $R$ such that if $(k, y)$ and $(k', y')$ are centered $\tau$-representations of $x,x'\in X$ respectively and $x_{[-R, R)} = x'_{[-R,R)}$, then $k=k'$ and $y_0 = y'_0$.
\end{enumerate}
\end{lem}
The constant $R$ provided by \cref{lem:mosserec} is called a {\em recognizability constant} for $\tau$.  

\begin{proof}
Since $\tau$ is non erasing, for every $a \in \cA$ the set $\tau([a])$ is clearly a non-empty closed subset of $X$. From the hypothesis $X$ is equal to $\cup_{a \in \cA} \cup_{0 \le k < |\tau(a)|} S^k\tau([a])$.

If the morphism $\tau$ is recognizable in $Y$, then the closed sets $S^k\tau([a])$ are pairwise disjoint. It follows that they are also open and hence they are clopen sets. 

If the set $\tau ([a])$ is a non-empty clopen set, then for each  integer $k$ the set $S^k\tau([a])$ is also a non-empty clopen set. This means that each $S^k\tau([a])$ is a finite union of cylinders $[u.v] = \{ x\in X : x_{[-|u| , |v|)} = uv\}$ for words $u$, $v$ in $\cB^*$. 
Without loss of generality we can assume $|u|=|v|=R$ for some integer $R$, for all the words $u$ and $v$ and $0\leq k < |\tau(a)|$. 
Hence, if  $x_{[-R, R)} = x'_{[-R,R)}$ for $x,x' \in X$, then $x$ and $x'$ belong to the same clopen set $S^k\tau([a])$ which proves \eqref{item:recradius}.

Clearly \eqref{item:recradius} implies the recognizability of $\tau$ in $Y$.
\end{proof}

Observe that \cref{lem:mosserec} states that the notion of recognizability for a morphism is local. 
Historically, a first notion of recognizability for substitutions appears with the work of Mossé \cite{Mosse:1992}.
Actually the recognizability of a non-erasing morphism does imply the recognizability in the sense of Mossé \cite{Berthe&Steiner&Thuswaldner&Yassawi:2018}. 

For composition of morphisms, in \cite{Berthe&Steiner&Thuswaldner&Yassawi:2018} the following result is proved.

\begin{lem}\cite[Lemma 3.5]{Berthe&Steiner&Thuswaldner&Yassawi:2018}\label{lem:CompMorphRec}
Let $\tau \colon \cA^* \to \cB^{\ast}$ and $\sigma \colon \cB^* \to \cC^{\ast}$ be non-erasing morphisms, $Y\subseteq \cA^{\Z}$  be a subshift and $X$ be the smallest subshift containing $\tau(Y)$. 
Then $\sigma \circ \tau$ is recognizable in $Y$ if and only if $\tau$ is recognizable in $Y$ and $\sigma$ is recognizable in $X$.  
\end{lem}

In the next lemma we show that a non-recognizable morphism  is the  composition of a recognizable one with a letter-to-letter morphism.

\begin{lem}\label{lem:DonosoTrick}
Let $Y \subseteq \cA^{\Z} $ be a subshift and $\tau \colon \cA^* \to \cB^*$ be  a non-erasing morphism. Then there exist an alphabet $\cC$, a morphism $\sigma \colon \cA^* \to \cC^*$ recognizable in $Y$ for aperiodic points and a map $\psi \colon \cC \to \cB$ such that 
\[ \tau = \psi \circ \sigma. 
\]
\end{lem}
\begin{proof}
Set $\cC = \cB \times \cA$. Let $\sigma$ and $\psi $ be the morphisms defined by 
\[
\sigma (a) = (1, a)\cdots (|\tau(a)|, a) \textrm{ and } \psi(b,a) = \tau (a)_b, \textrm{ for } b\in \cB, a\in \cA. 
\]
These maps clearly satisfy the required relation. 
%
\end{proof}

Inspired by the  definitions on $\cS$-adic subshifts in \cite{Berthe&Steiner&Thuswaldner&Yassawi:2018},  
let us define a   last notion of recognizability for a family of words with respect to a subshift.

Let $\cW\subset \cB^*$ be a finite set of non-empty words and let $X\subseteq \cB^\Z$ be a subshift. 
Consider a subshift $Y\subseteq \cA^\Z$ and a morphism $\tau: \cA^* \to \cB^*$ such that $\tau(\cA)=\cW$. In particular $\tau$ is non-erasing. 

We say $\cW$ is {\em recognizable in $X$ with respect to $(Y,\tau)$} if $X$ is the smallest subshift containing $\tau(Y)$,
({\it{i.e.}}, $X=\cup_{k\in\Z} S^k(\tau(Y))$) and $\tau$ is recognizable in $Y$, ({\it{i.e.}}, every $x\in X$ has a unique centered $\tau$-representation in $Y$). We say that $\cW$ is {\em recognizable in $X$} if it is recognizable in $X$ with respect to some subshift $Y$ and a morphism $\tau$ as before.  

We say that a $\cW$-factorization $F_x: \Z \to \cW\times \Z$ of $x\in X$ is {\em adapted to} $(Y,\tau)$ if there exists $y=(y_i)_{i\in \Z}\in Y$ such that $\tau(y_i)$ is the first coordinate of $F_x(i)$ for all $i\in \Z$.

The following lemma follows  straightforwardly from definitions.

\begin{lem}
\label{lem:fact-rec}
Let $\cW \subset \cB^*$ be a finite set of non-empty words and  $X\subseteq \cB^\Z$ be a subshift. Consider a subshift $Y\subseteq \cA^\Z$ and a morphism $\tau: \cA^* \to \cB^*$ such that $\tau(\cA)=\cW$. 
The following statements are equivalent: 
\begin{enumerate}
\item $\cW$ is recognizable in $X$ with respect to $(Y,\tau)$;
\item $X$ is the smallest subshift containing $\tau(Y)$ and every $x\in X$ has a unique $\cW$-factorization adapted to $(Y,\tau)$.
\end{enumerate}
\end{lem}

\cref{lem:mosserec} together with \cref{lem:fact-rec} imply the following local property for this last notion of recognizability.

\begin{lem}
\label{lem:rayon-rec}
Let $\cW \subset \cB^*$ be a finite set of non-empty words and let $X\subseteq \cB^\Z$ be a subshift. Consider a subshift $Y\subseteq \cA^\Z$ and a morphism $\tau: \cA^* \to \cB^*$ such that $\tau(\cA)=\cW$.
Assume $\cW$ is recognizable in $X$ with respect to $(Y,\tau)$ and let $R$ be a recognizability constant for $\tau$. 
Then, for all $x\in X$ the set $\{x' \in X:  F_{x'}(0)=F_x(0)\}$ is non-empty and clopen. Moreover, if $x, x' \in X$ are such that $x_{[-R,R)} = x'_{[-R,R)}$ then 
$F_x (0)=F_{x'}(0)$.
\end{lem}

\subsection{Recognizability for directive sequences}\label{subsec:recogforDirSequences}

A~directive sequence \hfill\break ${\boldsymbol{\tau}}=$ $(\tau_n\colon \cA_{n+1}^*\to\cA_n^*)_{n\geq 0}$ is \emph{recognizable at level~$n$} if $\tau_n$ is 
recognizable in~$X_{\boldsymbol{\tau}}^{(n+1)}$. 
The directive sequence~$\boldsymbol{\tau}$ is \emph{recognizable} if it is recognizable at level~$n$ for each $n \ge 0$. If there is an $N \geq 0$ such that $\boldsymbol{\tau}$ is recognizable at level~$n$ for each $n \ge N$, then we say that  $\boldsymbol{\tau}$ is \emph{eventually recognizable}. 

It is not difficult to prove 
using the fact that $\tau_n(X_{\boldsymbol{\tau}}^{(n+1)})$ is included in  $X_{\boldsymbol{\tau}}^{(n)}$ that the directive sequence 
$\boldsymbol{\tau}$ is recognizable if and only if for all $n\geq 0$ and $x \in X_{\boldsymbol{\tau}}(=X_{\boldsymbol{\tau}}^{(0)})$ there is at most one 
$y \in X_{\boldsymbol{\tau}}^{(n)}$ and $0\leq k < |\tau_{[0,n)}(y_0)|$ such that 
$x=S^k(\tau_{[0,n)}(y))$. 
In fact if $\boldsymbol{\tau}$ is recognizable, then all $x\in X_{\boldsymbol{\tau}}$ has a unique such couple $(k,y)$. 
The same result follows for the notion of eventually recognizable directive sequences.

Given a $\cS$-adic subshift with bounded alphabet rank, in \cite{Berthe&Steiner&Thuswaldner&Yassawi:2018} 
(inspired by Theorem \ref{DowMaass}) the authors show that there is a uniform bound for the number of 
morphisms in its directed sequence that are not recognizable for aperiodic points. Thus any $\cS$-adic subshift is forced to be eventually recognizable for aperiodic points. The bound in \cite{Berthe&Steiner&Thuswaldner&Yassawi:2018} can be improved by following the same steps of the original proof but considering the bound in Theorem \ref{theo:Wfact}.  

\begin{theo} 
\label{theo:bsty}
Let $\boldsymbol{\tau} = (\tau_n \colon  \mathcal{A}^*_{n+1}\to \mathcal{A}_n^*)_{n\ge 0}$ be a sequence of morphisms with $d = \liminf_{n\to\infty} |\mathcal{A}_n| < \infty$. 
Then, there are at most $\log_2 d$ morphisms $\tau_n$ that are not recognizable in 
$X_{\boldsymbol{\tau}}^{(n+1)}$ for aperiodic points.
In particular, $\boldsymbol{\tau}$ is eventually recognizable for aperiodic points. 
\end{theo}

\section{From finite topological rank minimal Cantor systems to $\cS$-adic subshifts and vice versa}
\label{section:finiterankvsSadic}

The relation between the classes of expansive minimal Cantor systems of finite topological rank and $\cS$-adic subshifts with bounded alphabet rank is not fully understood, at least from a topological perspective. Nevertheless, it is implicit in the works of \cite{Downarowicz&Maass:2008,Berthe&Steiner&Thuswaldner&Yassawi:2018,Durand&Leroy:2012} that both classes somehow coincide. 
Indeed, in \cite{Berthe&Steiner&Thuswaldner&Yassawi:2018} it is proved under mild conditions that measure theoretically this holds. Also, in \cite{Durand&Leroy:2012} conditions on a directive sequence are provided so that it generates a  $\cS$-adic subshift of finite topological rank.

The main purpose of this section is to prove this correspondence  in the topological setting modulo the recognizability of the directive sequence defining the $\cS$-adic system. 
In order to do that, we provide canonical constructions allowing to build a finite topological rank Bratteli-Vershik representation of a given recognizable $\cS$-adic subshift and on the other hand to define the directive sequence of a $\cS$-adic subshift conjugate to a given expansive finite topological rank minimal Cantor system. 

More precisely, the main theorem of this section is the following.

\begin{theo}\label{theo:equivalenceSadicvsFiniteRank}
\text{}
\begin{enumerate}
\item
\label{theo:partone}
An expansive finite topological rank minimal Cantor system is topologically conjugate to 
a primitive recognizable $\cS$-adic subshift with bounded alphabet rank. 
Moreover, the alphabet rank of the corresponding directive sequence is bounded by the rank of the minimal Cantor system.   

\item Any  primitive $\cS$-adic subshift (so minimal) with bounded alphabet rank is a topological factor of a finite topological rank minimal Cantor system whose rank is bounded by the square of the alphabet rank of the subshift. 
Moreover, when  the $\cS$-adic subshift is recognizable then the factor map is invertible and the subshift  has finite topological rank.  
\end{enumerate}
\end{theo}
By results in 
\cite{Berthe&Steiner&Thuswaldner&Yassawi:2018} one has that all primitive $\cS$-adic subshifts with bounded alphabet rank are eventually recognizable, so the factor map in (2) is just the one projecting the recognizable level to the given $\cS$-adic subshift. 
It is also interesting to comment that there are several $\cS$-adic subshifts, like Chacon subshift, that are not primitive but one can prove easily that they have finite topological rank (for Chacon example the topological rank is two). 

\subsection{Combinatorial lemmas to get finite topological rank systems}
\label{sec:CombiLemmas}

Before proving the main theorem of the section we develop two general combinatorial ideas that we think can be useful in different contexts.  

\subsubsection{Combinatorial constructions using recognizable sets of words}

In this section we give a sufficient condition for a minimal subshift to have finite topological rank. 
Let $X \subseteq \cB^\Z$ be a subshift and $\cW \subset \cB^*$ be a finite set of non-empty words that is 
recognizable in $X$  with respect to $(Y,\tau)$, where $Y\subseteq \cA^\Z$ is a subshift
and $\tau: \cA^* \to \cB^*$ is a morphism with $\tau(\cA)=\cW$.  
Then for each $x\in X$ there exists a unique $\cW$-factorization $F_x\colon  \Z \to \cW\times \mathbb{Z}$ adapted to $(Y,\tau)$. 

For $w=w^-w^+ \in \cW$ with $w^-,w^+ \in \cB^*$ we define the set
$$
\llbracket w^-.w^+ \rrbracket =\{x \in X \ : \ F_x(0)= (w,-|w^-|) \}.
$$
In words, $\llbracket w^-.w^+ \rrbracket$ is the set of points in $X$ whose $\cW$-factorizations adapted to $(Y,\tau)$ have 
$w$ as a centered word following the decomposition $w^-.w^+$, where the dot indicates the separation between negative and non-negative coordinates. If $w^-$ is the empty word we just write 
$\llbracket w \rrbracket$. Clearly $\llbracket w^-.w^+ \rrbracket = S^{|w^-|} \llbracket w \rrbracket$.
Given another word $v\in \cW$ we also set
$$
\llbracket v , w^-.w^+ \rrbracket =\{x \in X \ : \ F_x(-1)=(v,-(|v|+|w^-|)), \ F_x(0)= (w,-|w^-|) \}.
$$

The following result follows from \cref{lem:rayon-rec}. 

\begin{lem}\label{lem:localrec}
Consider $\cW \subset \cB^*$, $X\subseteq \cB^\Z$, $Y\subseteq \cA^\Z$ and $\tau:\cA^*\to \cB^*$ as before. Assume $\cW$ is recognizable in $X$ with respect to $(Y,\tau)$ and let $R$ be a recognizability constant for $\tau$. 
Then, for all $w=w^-w^+, v \in \cW$ the sets
$\llbracket w^-.w^+ \rrbracket$ and $\llbracket v, w^-.w^+ \rrbracket$ are clopen, thus finite unions of cylinder sets of $X$: 
$$
\llbracket v, w^-.w^+ \rrbracket = \bigcup_{x\in \llbracket v, w^-.w^+ \rrbracket} [x_{[-R-|vw^-|, 0)}.x_{[0,  |w^+| + R)}].
$$
\end{lem}

The next theorem is the main construction of this subsection. 

\begin{theo}[Sufficient condition for finite topological rank]\label{lem:sufforfiniterank}
Let $(X,S)$ be a minimal subshift. 
Then its topological rank is bounded from above by $K(X)^2$, where  
$$
K (X) = \lim_{n\to +\infty}  \inf_{\substack{\cW \subset \cup_{k\geq n}\cL_k (X) \\  
\cW \ {\rm is \ recognizable \ in } \ X}}  |\cW|.
$$
\end{theo}

\begin{proof}
We assume that $K=K(X)$ is finite, which is the unique relevant case. Under this assumption 
there exists a sequence $(\cW_n)_{n\geq 1}$ of subsets of $\cB=\cL(X)$ such that  
$|\cW_n|=K$, $\lim_{n\to +\infty} \inf_{w \in \cW_n} |w| = +\infty$, and $\cW_n$ is recognizable in $X$ 
with respect to a subshift $Y_n \subseteq \cA_{n}^\Z$ and a morphism $\tau_n:\cA_n^*\to \cB^*$ where $\tau_n(\cA_n)=\cW_n$.
For all $n\geq 1$ let $R_n$ be a recognizability constant for $\tau_n$.
We can assume
that  the sequence $(R_n)_{n\geq 1}$ is increasing. 

Now we will suppose, taking a subsequence of $(\cW_n)_{n\geq 1}$ if needed, that all $w\in \bigcup_{n\geq 1} \cW_n$ decomposes  as $w=w^- w^+$ in such a way that: 
\begin{enumerate}
\item
\label{item:croit} for all $n\geq 2$, 
$\max_{w \in \cW_{n-1}}(|w^-|,|w^+|)<\min_{w\in \cW_n}(|w^-|,|w^+|)/2$;
\item for all $n\geq 2$, 
\label{item:versunpoint}
$$
\bigcup_{w\in \cW_{n}} \llbracket w^-.w^+ \rrbracket \subseteq 
\llbracket v_{n-1} , w_{n-1}^-.w_{n-1}^+ \rrbracket
\cap [l_{n-1} v_{n-1} w_{n-1}^- . w_{n-1}^+ r_{n-1}]
$$
for some $w_{n-1}, v_{n-1}  \in \cW_{n-1}$ and $l_{n-1}, r_{n-1} \in \mathcal{L}_{R_{n-1}} (X)$;
\item \label{item:partition}
for all $n\geq 1$, the set 
$$
\cT_n
=
\{S^{-i} \llbracket v,w^-.w^+ \rrbracket: v,w \in \cW_n ,0\leq i < |w^-|+|v^+|
\}
$$
is a partition of $X$.
\end{enumerate}
First we check that once we have a sequence $(\cW_n)_{n\geq 1}$ with these properties, then the sequence of partitions $(\cT_n)_{n\geq 0}$, with $\cT_0=\{X\}$, is a nested sequence of CKR partitions, ({\it i.e.}, it satisfies (KR1)-(KR4)). 
This will prove the theorem. 
After we will show how to build such a sequence $(\cW_n)_{n\geq 1}$.

Let $n\geq 1$. We order in a particular way the sets in $\cT_n$ to get the structure of Kakutani-Rohlin towers. 
Each tower is indexed by pairs of words $v,w \in \cW_n$ such that $vw$ appears in a $\cW_n$-factorization in $X$ adapted to $(Y_n,\tau_n)$. 
For such $v,w \in \cW_n$ the associated tower is given by: 
$$
\cT_n(v,w)=\{S^{-i} \llbracket v,w^-.w^+ \rrbracket \ : \ 0\leq i < |w^-|+|v^+| \}.
$$
In the nomenclature of Kakutani-Rohlin partitions the base of this tower is given by 
$\llbracket v,w^-.w^+ \rrbracket$. Observe that this set is clopen by \cref{lem:localrec}.
Thus we have defined at most $K^2$ towers per level. 

By Property \eqref{item:versunpoint} the sequence of towers satisfies \ref{item:KR1}.

Now we prove \ref{item:KR2} holds. 
Let $n\geq 2$, $v,w \in \cW_n$ and $i\in [0 ,|w^-| + |v^+| )$ be an integer.  
From property \eqref{item:versunpoint}, since the length of $l_{n-1}$ and $r_{n-1}$ are $R_{n-1}$, all elements $x,x' \in S^{-i} \llbracket vw^-.w^+ \rrbracket$ satisfy 
$x_{[-R_{n-1}, R_{n-1})}= x'_{[-R_{n-1}, R_{n-1})}$.
Hence the unicity of the factorization into $\cW_{n-1}$ words (\cref{lem:rayon-rec})  and Property \eqref{item:croit} imply there is a sequence of  words  $u_1, \ldots, u_{\ell} \in \cW_{n-1}$ with $u_1 =u_{\ell}= w_{n-1}$
such that any set $S^{-i} \llbracket v,w^-.w^+ \rrbracket$ in
$\cT_n(v,w)$ is contained in some $S^{-j} \llbracket u_k,u_{k+1}^-.u_{k+1}^+ \rrbracket \in 
\cT_{n-1}(u_k,u_{k+1})$ for some $0\leq k \leq \ell$, with $u_0= v_{n-1}$ (Property \eqref{item:versunpoint}).
This proves \ref{item:KR2} and provides  
$
v^+w^-=u_1^+u_2\cdots u_{\ell-1}u_{\ell}^-.$

Property \ref{item:KR3} is deduced from properties  \eqref{item:croit} and \eqref{item:versunpoint}. Indeed, Property \eqref{item:versunpoint} implies that for $n\geq 2$ and $v, w \in \cW_n$ as before,
$$
\llbracket v,w^-.w^+ \rrbracket \subseteq \llbracket w^-.w^+ \rrbracket 
\subseteq \llbracket v_{n-1}, w_{n-1}^-.w_{n-1}^+ \rrbracket,
$$
for fixed words $v_{n-1},w_{n-1} \in \cW_{n-1}$, and  Property  \eqref{item:croit} implies that 
the lengths of $w_{n-1}^-$ and $w_{n-1}^+$ grow to infinity with $n$. 

Finally, to show \ref{item:KR4} it is enough to observe from Property \eqref{item:versunpoint} that each atom of ${\mathcal T}_n$ of the form $S^{-i}\llbracket v, w^-.w^+\rrbracket$ is included into a classical cylinder set of the form $[u.u']$ for words $u$ and $u'$ whose lengths are greater than $\min\{|w_{n-1}^+|, |w_{n-1}^-|\}$ (going to infinity with $n$  by Property \eqref{item:croit}). So the diameter of each atom of  ${\mathcal T}_n$ decreases to $0$ which shows \ref{item:KR4}. 
\smallskip

Now we describe  how to construct  from the  sequence of  finite sets of  words $(\cW_n)_{n\geq 1}$ at the beginning of the proof a sub-sequence $(\cW_{n_i})_{i\geq 1}$ satisfying Properties \eqref{item:croit}-\eqref{item:partition}. 
This is done recursively. 

Let $n_1>0$ be an integer such that $\inf_{w \in \cW_{n_1}} |w| >2$. 
For any word $w \in \cW_{n_1}$, we choose an arbitrary decomposition $w = w^-w^+$ such that $|w^+|$,$|w^-| \ge |w|/2-1$.
We only have Property \eqref{item:partition} to check. It is obviously done by the 
recognizability of the set $\cW_{n_1}$ with respect to $(Y_{n_1},\tau_{n_1})$.

Let us assume we have constructed $\cW_{n_1},\ldots, \cW_{n_\ell}$ with the desired properties. 
We choose two arbitrary words $v_{n_\ell}$ and $w_{n_{\ell}}$  in $\cW_{n_{\ell}}$ such that $v_{n_\ell}w_{n_{\ell}}$ appears in a  $\cW_{n_\ell}$-factorization of $X_{n_\ell}$ adapted to $(Y_{n_\ell},\tau_{n_\ell})$.
Let $l_{n_\ell}$ and  $r_{n_\ell}$ be two words of length $R_{n_\ell}$ such that $l_{n_\ell}v_{n_\ell}w_{n_\ell}r_{n_\ell}$ belongs to  ${\mathcal L}(X)$ and $[l_{n_\ell}v_{n_\ell}w_{n_\ell}^-.w_{n_\ell}^+r_{n_\ell}] \subset \llbracket v_{n_\ell},w_{n_\ell}^-.w_{n_\ell}^+\rrbracket$ (\cref{lem:rayon-rec}).

By minimality of $(X,S)$, since the lengths of the words of $\cW_n$ grow to infinity,  for any large enough integer $n$ (greater than some $N_\ell$), any word $w$ in $\cW_n$ contains an occurrence of $l_{n_\ell}v_{n_\ell}w_{n_\ell}r_{n_\ell}$.  Choosing such an occurrence leads to decompose $w =w^-w^+$, where the words $w^-$  and $w^+$  contain respectively  the word $l_{n_\ell}v_{n_\ell}w_{n_\ell}^-$  and $w_{n_\ell}^+r_{n_\ell}$ as a suffix and as a prefix respectively. 
Moreover, taking $n$ large enough, we can also assume that the occurrence is far away from the border of $w$  so that 
$$\max_{w \in \cW_{n_{\ell}}}(|w^-|,|w^+|)<\min_{w\in \cW_n}(|w^-|,|w^+|)/2.$$
We choose one integer $n_{\ell +1}$ realizing this. The last inequality ensures Property \eqref{item:croit} and the construction ensures Property \eqref{item:versunpoint}. 

To show $\cT_{n_{\ell+1}}$ is a partition observe that given $x\in X$, the recognizability of $\cW_{n_{\ell+1}}$ in $X$ with respect to $(Y_{n_{\ell+1}},\tau_{n_{\ell+1}})$ implies  there are unique 
words in $v, w\in \cW_{n_{\ell+1}}$ and $0\leq i < |w^-|+|v^+|$ such that 
$S^i x $ belongs to $ \llbracket v, w^-.w^+ \rrbracket$. Hence $\cT_{n_{\ell+1}}$ covers the space $X$. 
By construction all the sets $S^{-i} \llbracket v, w^-.w^+ \rrbracket$ are pairwise disjoint.
\end{proof}

We deduce the following theorem.

\begin{theo}\label{theo:FRankSuff}
Let $(X,S)$ be a minimal subshift. It has finite topological rank if and only if $K(X)$ is finite.
\end{theo}
\begin{proof}
The sufficient condition follows directly from \cref{lem:sufforfiniterank}. 

To prove the necessity of the condition assume $(X,S)$ has finite topological rank and that it is given by a Bratteli-Vershik representation $B=(V,E,\preceq)$ where the number of vertices per level is uniformly bounded by the rank.
Moreover, by Theorem \ref{DowMaass}, after telescoping the diagram and because $(X_B , V_B)$ is expansive, we can assume $(X_B,V_B)$ is topologically conjugate to the subshift $(X^{(0)}_B,S)$, $X_B^{(0)}\subseteq E_1^{\Z}$, defined by coding orbits of the Bratteli-Vershik representation on the first level of the diagram via the map $\phi_0$ (we follow  the notations in Section \ref{sec:BVrepresentation}).

We claim it is enough to show that $K(X_B^{(0)})$ is finite. 
Indeed, assume there is a sequence $(\cW_{n})_{n\geq 1}$ of subsets of $\cB=\cL(X_B^{(0)})$ such that  for all $n\geq 1$, $|\cW_{n}| \leq K(X_B^{(0)}) < + \infty$, $\lim_{n\to +\infty} \inf_{w \in \cW_{n}} |w| = +\infty$ and 
 $\cW_{n}$ is recognizable in $X_B^{(0)}$ with respect to a subshift $Y_{n} \subseteq \cA_{{n}}^\Z$ and a morphism $\tau_n:\cA_n^*\to \cB^*$ with $\tau_n(\cA_n)=\cW_n$.
 By the Curtis-Hedlund-Lyndon theorem, the conjugacy map $\phi \colon X_B^{(0)} \to X$ is given by a {\em block code} $\hat{\phi} \colon \cL_{2r+1}(X^{(0)}) \to \cL_{1}(X)$  of some radius $r$: for all $i\in \mathbb{Z}$
 $$
 \phi (x)_i = \hat{\phi} (x_{[i-r , i+r]}) .
 $$
 For any large enough $n$, the length of any word in $\cW_{n}$ is larger than $2r+1$. 
Consider a triple of words $u,w,v \in \cW_{n}$ such that $uwv$ appears in a $\cW_{n}$-factorization in $X_B^{(0)}$ adapted to $(Y_{n},\tau_{n})$. Its  image $\hat{\phi}(uwv)$ by the local map  is of the form $\hat u \hat w \hat v \in \cL_{|u|-r+|w|+|v|-r}(X)$ with $|\hat w|=|w|$, $|\hat u|=|u|-r$ and $|\hat v|=|v|-r$. Thus we define $\hat \cW_{n}$ to be the collection of  words $\hat w$ built in this way. It is clear from construction that such sets of words have cardinality at most $K(X_B^{(0)})^3$ and since the inverse of the map $\phi$  is 
also given by a block code,
the set $\hat \cW_{n}$ is recognizable in $X$ with respect to $(Y_n,\tau_n)$. 
It follows that $K(X)$ is finite.

Now we produce from the Bratteli-Vershik representation $B$ the sets of words $(\cW_{n})_{n\geq 1}$ as described before. 
Recall that $X_B^{(0)}$ is the $\cS$-adic subshift generated by the sequence $\boldsymbol{\tau}^B= (\tau_n^B:V^*_{n+1}\to V_n^*)_{n\geq 0}$ of morphisms read on $B$.
For each $n\geq 1$ we define $\cW_n$ to be the set of words $\tau^B_{[0,n)}(v) \in E_1^\Z$  with $v \in V_{n}$. 
Clearly the cardinality of each $\cW_n$ is bounded by the topological rank of $(X,S)$.
Recall that the map $\phi_0 \colon X_B \to X_B^{(0)}$ is a homeomorphism.   
It is simple to check that the preimage by $\phi_0$ of the sets 
$\tau^B_{[0,n)}([v])$, $v \in V_{n}$, are the roofs of towers of level $n$  in the Bratteli-Vershik representation of the system and in particular are clopen sets. Thus, by Lemmas \ref{lem:mosserec} and \ref{lem:fact-rec}, the set $\cW_n$ is recognizable in $X_B^{(0)}$ with respect to $(X_B^{(n)},\tau^B_{[0,n)})$.
Since the sequence $\boldsymbol{\tau}^B$ is primitive, the lengths of the words in $\cW_n$ grow with $n$ to infinity. This construction shows that $K(X_B^{(0)})$ is finite and thus the theorem holds.
\end{proof}

\subsection{Proof of Theorem \ref{theo:equivalenceSadicvsFiniteRank} }

We have all the elements to prove Theorem \ref{theo:equivalenceSadicvsFiniteRank}.  
Actually  Part \eqref{theo:partone} of Theorem \ref{theo:equivalenceSadicvsFiniteRank} comes from a property  (\cref{theo:equivalenceSadicvsFiniteRank2}) available   for Cantor minimal systems of  eventually infinite rank. 
This proposition provides a necessary condition on a Bratteli-Vershik system to be expansive in terms of recognizability of the sequence of morphisms read on the diagram. This condition is not too far from being sufficient, but the morphisms need to be proper  to get the converse (see \cref{prop:DurandLeroy2.3}). The results are presented in  Subsection \ref{sec:FRtoSadic}. 
The second subsection is devoted to the second part of Theorem \ref{theo:equivalenceSadicvsFiniteRank} stated as \cref{prop:SadictoFR}.  

\subsubsection{From finite topological rank systems to $\cS$-adic systems}\label{sec:FRtoSadic}

First, we recall a useful property to identify a Bratteli-Vershik system with a $\cS$-adic subshift.
We call a {\em hat-morphism} any morphism $\tau \colon \cA^* \to \cB^*$ such that for all $a,b \in \cA$ the letters occurring in $\tau(a)$ and $\tau(b)$ are all distinct. For instance the morphism read at the first level on an ordered Bratteli diagram is a  hat-morphism (see Section \ref{sec:BVrepresentation}). 
The next result is a slight modification of \cite[Proposition  2.2]{Durand&Leroy:2012}.

\begin{prop}[{\cite[Proposition  2.2]{Durand&Leroy:2012}}]\label{prop:DurandLeroy2.3}
Let $\boldsymbol{\tau}= (\tau_{n}:\cA_{n+1} \to \cA_n)_{n\ge 0}$ be a recognizable  directive sequence of  morphisms  such that  $\tau_0$ is a hat-morphism and $(\tau_{n})_{n\ge 1}$ is proper. 
Then, the $\cS$-adic subshift generated by $\boldsymbol{\tau}$ is topologically conjugate to the Bratteli-Vershik system $(X_{B}, V_{B})$ where $B$ is the Bratteli diagram such that $\boldsymbol{\tau}$ is the sequence of morphisms read on $B$.
\end{prop} 

Here is a converse to
\cref{prop:DurandLeroy2.3}.

\begin{prop}\label{theo:equivalenceSadicvsFiniteRank2}
Let $(X,T)$ be a minimal Cantor system given by a Bratteli-Vershik representation $B$. If $(X,T)$ is expansive, then the sequence of morphisms  $\boldsymbol{\tau}^B = (\tau_n^B:V_{n+1}^*\to V_n^*)_{n\geq 0}$  read on $B$ is eventually  recognizable. 
Moreover, $(X,T)$ is topologically conjugate to the $\cS$-adic subshift generated by a proper recognizable directive sequence obtained from $\boldsymbol{\tau}^B$ after a telescoping of $B$.
\end{prop}
\begin{proof}
Let us assume that  $(X,T)$ is an expansive minimal Cantor system. Then, the system is topologically conjugate to $(X_B^{(N)},S)$ for some $N\geq 0$. Thus, by telescoping finitely many levels of the diagram $B$, we can assume that $N=0$, so $(X,T)$ is topologically  conjugate to $(X_B^{(0)},S)$ (see Section \ref{sec:BVrepresentation} for notations). In addition, by simplicity of $B$, after another telescoping we can assume that the directive sequence $\boldsymbol{\tau}^B =(\tau_n^B:V_{n+1}^*\to V_n^*)_{n\geq 0}$ of morphisms read on $B$ is proper. 
To conclude we only need to prove that 
$\boldsymbol{\tau}^B$ is recognizable. The result follows after interpreting the telescoping process. 

Recall that under our assumptions for each $n\geq 0$ the subshift $(X_B^{(n)},S)$ is minimal and $X_B^{(n)}$ is the smallest subshift containing $\tau^B_n(X_B^{(n+1)})$. This implies that $X_B^{(0)}$ is the smallest subshift containing $\tau^B_{[0,n+1)}(X_B^{(n+1)})$. 
Since $\tau^{B}_{[0,n+1)}([a])$ is a clopen set for any letter $a \in  V_{n+1}$ (just consider its preimage by the homeomorphism $\phi_0:X_B\to X_B^{(0)}$), then  
by \cref{lem:mosserec} we get that the morphism $\tau^{B}_{[0,n+1)}$ is recognizable in $X_B^{(n+1)}$. This property and a recursive use of \cref{lem:CompMorphRec} proves that $\boldsymbol{\tau}^B$ is recognizable. 
\end{proof}

\subsubsection{From recognizable $\cS$-adic subshifts to finite topological rank minimal Cantor systems} 

For $\cS$-adic subshifts generated by non necessarily proper morphisms, we have the following dynamical property. 

\begin{prop}\label{prop:SadictoFR}
Let $\boldsymbol{\tau}$ be a primitive  directive sequence 
with finite alphabet rank $AR(\boldsymbol{\tau}) < + \infty$. 
Then, the $\cS$-adic subshift generated by $\boldsymbol{\tau}$ is a factor of a finite topological rank minimal Cantor system whose rank is bounded by $AR(\boldsymbol{\tau})^{2}$.
Moreover, when the $\cS$-adic subshift is recognizable then the factor map is invertible.
\end{prop}

\begin{proof} 
Let $\boldsymbol{\tau} = (\tau_n \colon \mathcal{A}^*_{n+1}\to \mathcal{A}_n^*)_{n\ge 0}$.
If the $\cS$-adic subshift $X_{\boldsymbol{\tau}}$ generated by  $\boldsymbol{\tau}$ is  finite, the result is classical since any  minimal periodic system occurs as the factor of an odometer (a system with topological rank one).  
When $X_{\boldsymbol{\tau}}$ is infinite,  Theorem \ref{theo:bsty} implies that $\boldsymbol{\tau}$ is eventually recognizable. This means that for some $N\geq 1$, the directive sequence $(\tau_n)_{n\geq N}$
is recognizable. They generate the $\cS$-adic subshift $X_{\boldsymbol{\tau}}^{(N)}$.
Let $\psi$ and $\sigma$ be the morphims provided by \cref{lem:DonosoTrick}  when applied  to $\tau= {\tau}_{[0,N)}$.  
Hence  the subshift $X_{\boldsymbol{\tau}}$ is a factor, via the map induced by $\psi$, of the $\cS$-adic subshift $X_{\boldsymbol{\tau'}}$ generated by the recognizable directive sequence $\boldsymbol{\tau'} = (\sigma, \tau_N, \tau_{N+1}, \ldots)$.
Thus the set of words $\cW_n=\{\tau'_{[0,n+1)}(a): a\in \cA_{n+1}\}$ for all $n\geq 0$ is recognizable in $X_{\boldsymbol{\tau'}}=X_{\boldsymbol{\tau'}}^{(0)}$ with respect to $(X_{\boldsymbol{\tau'}}^{(n+1)},\tau'_{[0,n+1)})$.
Since $\boldsymbol{\tau'}$ is primitive, the lengths of words in $\cW_n$ grow to infinity with $n$, so that $K(X_{\boldsymbol{\tau'}})$ defined in \cref{lem:sufforfiniterank} is bounded by $AR(\boldsymbol{\tau})$. This implies it has finite topological rank and that its topological rank is bounded by $K(X_{\boldsymbol{\tau'}} )^2$, showing the result.
\end{proof}

\section{non-superlinear complexity subshifts have finite topological rank}\label{sec:StrongDec}
In this section we show that subshifts with non-superlinear complexity have finite topological rank. Even better, we prove this also holds under the weaker assumption that the complexity function satisfies $\liminf_{n\to +\infty} p_X(n+1)-p_X (n) <+\infty$  (\cref{cor:NonSuperLinareFR}). 
To do this, we follow an idea originally due to T. Monteil to interpret complexity conditions in terms of graph properties. Thus we  characterize minimal subshifts with non-superlinear complexity in terms of deconnectability of the associated Rauzy graphs.   
Next we use the characterization  of subshifts of finite topological rank provided in \cref{theo:FRankSuff} to prove \cref{cor:NonSuperLinareFR}. 
Notice that the converse of  \cref{cor:NonSuperLinareFR} does not hold as illustrated by the example in  Section \ref{sec:ComplexityFR}. Let us also mention that in Section \ref{sec:ComplexityFR} we present sufficient conditions so that a finite topological rank subshift has a non-superlinear complexity.

\subsection{Subshifts with non-superlinear complexity and Rauzy Graphs}\label{subsec:RauzyGraph}

We start with a combinatorial lemma on graphs. Next we apply this lemma to Rauzy graphs associated to subshifts to get bounds for the complexity of the subshift.  
For a directed graph $G=(V,E)$ we denote by $\textrm{d}^+(v)$ (resp. $\textrm{d}^-(v)$) the outdegree (resp. indegree)  of a vertex $v \in V$. Recall that a (directed)  {\em path} in $G$ is a finite sequence of distinct edges $e_1,\ldots, e_\ell \in E$ (oriented in the same direction) which joins a sequence of distinct vertices ($e_i =(v_i,v_{i+1})\in E$ and $e_{i+1}= (v_{i+1}, v_{i+2})\in E$).
Given a subset $V' \subseteq V$ of vertices, we say that the path is {\em in the complementary of $V'$ except the border} whenever $v_1, v_{\ell+1}$ are in $V'$ and $v_i \not\in V'$ for $i\in \{2,\ldots,\ell\}$. A {\em tree} is an unoriented acyclic connected graph and 
a {\em forest}  is  a disjoint union of trees. 
Recall that a {\em directed forest} is a directed acyclic  graph  whose underlying undirected graph is a forest. 

\begin{lem}\label{lem:Monteil}
Let $G= (V,E)$ be a strongly connected finite directed graph and let $V'$ be the set of vertices $v \in V$ such that $\hbox{\rm d}^+ (v) \geq 2$.
Assume $V'$ is non-empty. 
Then,
\begin{itemize}
\item the subgraph induced by the set of vertices $V\setminus V'$ is a directed forest; 

\item the number of  directed paths in $G$ in the complementary of $V'$ except the border, is bounded from above by $|V'|^2 \cdot (\max_{v \in V}  \hbox{\rm d}^+ (v))
\cdot (\max_{v \in V}  \hbox{\rm d}^- (v))$.  
\end{itemize}
\end{lem}

\begin{proof}
Let $G[V\setminus V']$ denote the subgraph induced by the set of vertices $V\setminus V'$. 
 The graph $G[V\setminus V']$ has no directed cycles.
This follows noticing that if there is a directed cycle in it, then the strong connectivity of $G$ implies that one of the vertex in the cycle should  have at least two outgoing edges: this contradicts the definition of $G[V\setminus V']$.
There is actually no undirected cycles in it.  
Otherwise, a simple undirected cycle which is not a directed cycle necessarily has a vertex whose outdegree is at least two, leading again to a  contradiction. To prove this fact, observe that restricted to such cycle,
the sum of the outdegrees and the sum of the indegrees (for the induced orientation) are both equal to the number of edges which is equal to the number of vertices. 
It follows that each indegree and outdegree equals $1$ and the cycle is a directed cycle: a contradiction. Hence the graph $G[V\setminus V']$ is a directed forest. 

To bound the number of directed paths in $G$ in the complementary of $V'$ except the border, observe that such a path is uniquely determined by its starting edge (an outgoing edge of one element of $V'$) and  its ending edge (an ingoing edge of an element of $V'$). Indeed, given these choices, the path determines a path $\gamma$ in $G[V\setminus V']$, eventually empty if the  starting and ending edges are the same. Since $G[V\setminus V']$ is a directed forest, the path $\gamma$ is uniquely determined by the starting and ending edges. 
\end{proof}

Let $(X,S)$ be a subshift. 
The {\em Rauzy graph of $(X,S)$ at level $n\geq 2$} is the directed graph $G_n=(V_n,E_n)$ with  vertex set $V_n=\cL_{n}(X)$, where the ordered pair of words $(u_0\ldots u_{n-1},v_0\ldots v_{n-1})$ is an edge in $E_n$ if  $u_1\cdots u_{n-1} = v_0\cdots v_{n-2}$.
The indegree and outdegree of each vertex are bounded from above by the cardinality of the alphabet.
Moreover, when the subshift $(X,S)$ is minimal the Rauzy graph $G_n$ is strongly connected.  

For a subset $V'_n \subset V_n=\cL_n(X)$ let $G_n[V_n \setminus V'_n]$ denote the subgraph induced by the set of vertices $V_n \setminus V'_n$.
Our target for $V'_{n}$ will be the set of vertices in $V_n$ with outdegree greater than $1$, {\em i.e.}, the set of right special words of length $n$.

Using \cref{lem:Monteil} we propose a  combinatorial caracterization of subshifts with non-superlinear complexity in terms of Rauzy graphs.  It uses a refinement of the  notion of $K$-deconnectability introduced by  Monteil and developed together with Ferenczi in \cite{Ferenczi&Monteil:2010} to improve results of  \cite{Boshernitzan:1984}. We call this notion \emph{strong $K$-deconnectability}. 

\begin{defi}
Let $K\geq 1$ be an integer.
We say that a minimal subshift $(X,S)$ is {\em$K$-deconnectable} if there exist a sequence of non-negative integers $(n_i)_{i\geq 1}$ and a constant $K'\geq 1$ such that for all $i\geq 1$ there exists a subset $V'_{n_i} \subset \cL_{n_i} (X) $ of at most $K$ vertices such that  every path in $G_{n_i}[\cL_{n_i}(X)\setminus V'_{n_i}]$  has a length at most $K' n_i$.
If we further require that  $G_{n_i}[\cL_{n_i}(X)\setminus V'_{n_i}]$ is a directed forest,  we say that $(X,S)$ is {\em strongly $K$-deconnectable}.
\end{defi}
 
The next proposition links the combinatorial properties of the Rauzy graph with the asymptotic behavior of the complexity function. 
\begin{prop}\label{prop:Monteil}
Let $(X,S)$ be a minimal subshift. Then, $X$ has non-superlinear complexity  if and only if $(X,S)$ is strongly $K$-deconnectable for some $K\in \mathbb{N^{*}}$. 
\end{prop}
%
%
\begin{proof} 
One implication is implicit in \cite[Section 7.3.2]{Ferenczi&Monteil:2010} and we repeat the argument here for completeness. Suppose that $\liminf_{n\to +\infty} p_X(n)/n < +\infty$. We can find a constant $K$ and a sequence of positive integers $(n_i)_{i\geq 1}$ such that simultaneously $p_X(n_i)/n_i\leq  K$ and $s(n_i)=p_X(n_{i}+1)-p_X(n_i)\leq K$ (see for instance Boshernitzan \cite[Theorem 2.2]{Boshernitzan:1984}). So the number of right  special words of length $n_i$ is bounded by $K$ for all $i\geq 1$. We may remove all these special words $V'_{n_i}$ from the Rauzy graph $G_{n_i}=(\cL_{n_i}(X),E_{n_i})$ and we obtain the graph
$G_{n_i}[\cL_{n_i}(X)\setminus V'_{n_i}]$. By \cref{lem:Monteil} it is a directed forest and $(X,S)$ is  strongly $K$-deconnectable.

Conversely, let $(n_i)_{i\geq 1}$ be a sequence of positive integers and let $V'_{n_i}\subseteq \cL_{n_i}(X)$ be given by the strong $K$-deconnectability of $G_{n_i}$. Recall that the indegree and outdegree of any vertex in all Rauzy graphs $G_{n_i}$ are bounded by the cardinality $|\mathcal{A}|$  of the alphabet of $X$. As a consequence, the number of extremal vertices (that is, vertices with outdegree or indegree equal to 0)
in  the subgraph  $G[\cL_{n_i}(X)\setminus V'_{n_i}]$ is not greater than $2K|\mathcal{A}|$. 
Observe that in $G[\cL_{n_i}(X)\setminus V'_{n_i}]$, between two given extremal vertices there is at most one path because there are no cycles and all outdegrees are equal to $0$ or $1$.
Hence the number of paths between two extremal vertices in $G[\cL_{n_i}(X)\setminus V'_{n_i}]$ is bounded by  $2K|\mathcal{A}| -1$, uniformly in $i$. 
Then, a word of length $n_i$ being either of outdegree at least  $2$ or being a vertex of some path between two extremal vertices, one obtains:

\[
p_{X} (n_i) = | \cL_{n_i}(X) | \leq  
K+ 
\sum_{
\overset{\gamma:\text{path between extremal vertices } }{\text{ in } G[\cL_{n_i}(X)\setminus V'_{n_i}]}}  
\textrm{length}(\gamma) \leq K+ (2K|\mathcal{A}|-1)
\cdot K'n_i,  \]  
where $\textrm{length}(\gamma)$ stands for the length of the path $\gamma$. This finishes the proof.
\end{proof}

\subsection{Subshift of return words} Let $(X,S)$ be a minimal subshift. Let $\cW\subseteq \cL(X)$ be a finite set of non-empty words. In this paper we will only  consider the case where all words in $\cW$ have the same length $\ell_{\cW}$. A {\em first return word}  between  $w_1$ and $w_2 \in \cW$ (eventually identical)  is a word $u\in \cL(X)$ such that:
\begin{itemize}
\item[{\it i)}] $uw_2$ belongs to $\cL(X)$; 
\item[{\it ii)}] the word $w_1$ is a prefix of $uw_2$ and $uw_2$ has no other  occurrences of words in  $\cW$ than these two.
\end{itemize}

Let $\cR(\cW)$ be the collection of all first return words between words in $\cW$. The minimality of $(X,S)$ implies that $\cR(\cW)$ is a finite set. 

It will be convenient to label the return words. For this, consider $R(\cW)= \{1, \ldots, |\cR(\cW)|\}$ and let $\tau_{\cW} \colon R(\cW) \to \cR(\cW)$ be a bijection. The map $\tau_{\cW}$ induces a morphism 
$R(\cW)^{\ast} \to  \cR(\cW)^{\ast}$ also denoted by $\tau_\cW$. We call it a 
{\em coding morphism associated to} $\cW$.

Consider the set $Y_{\cR(\cW)}$ of all points $y\in R(\cW)^{\Z}$ such that 
$\tau_{\cW}(y)$ belongs to $X$ and
\begin{align*}
&  \{ i\in \Z : \tau_\cW (y)_{[i, i+ \ell_\cW) } \in \cW  \} \\
= & \{ |\tau_\cW (y_0\cdots y_\ell)|: \ell \ge 0\} \cup \{ - |\tau_\cW (y_\ell\cdots y_{-1})|: \ell \le -1    \} \cup \{0\} .
\end{align*}
It is closed and shift-invariant. 
We call $(Y_{\cR(\cW)} , S)$ the {\em subshift of return words to} $\cW$ {\em in} $X$.

Actually, the morphism $\tau_{\cW}$ is recognizable in $Y_{\cR(\cW)}$ (see Section \ref{subsec:recogformorphisms} for background on recognizability).

\begin{lem} \label{lem:RecognizabilityReturnWords}
Let $(X,S)$ be a minimal subshift and $\cW\subset
\cL(X)$ be a finite set of non-empty words of the same length $\ell_\cW$. Then, $\tau_\cW$ is recognizable in $Y_{\cR(\cW)}$. 
\end{lem}

\begin{proof}
Let $R$ be greater than $2 \ell_\cW + 2 \sup_{u \in \cR(\cW)} |u|$.  According to \cref{lem:mosserec}, it is enough to show for  $x,x' \in X$ and  $(k, y)$, $(k', y')$ two centered $\tau_\cW$-representations of $x,x'$ that  $k= k'$ and $y_0 = y'_0$ whenever $x_{[-R,R)} =x'_{[-R,R)}$.
By the choice of the constant $R$, there exist at least three occurrences of words belonging to $\cW$ in  $x_{[-R,R)}$.  
The definition of elements of $Y_{\cR(\cW)}$ together with  item {\it ii)} in the definition of first return words enable us to conclude. 
\end{proof}


Now we are ready to deduce that a minimal subshift with non-superlinear complexity has finite topological rank. In fact we prove a stronger result. We only need to consider systems where 
$\liminf_{n\to +\infty} {p_X(n+1)-p_X(n)}$ is finite, which ensures the conclusions of \cref{lem:Monteil} for infinitely many Rauzy graphs of the system. This last property is related but weaker than the strong $K$-deconnectability one. It is nevertheless enough to give a uniform bound on the number of return words to the (right) special words.

\begin{theo}\label{cor:NonSuperLinareFR}
Let $(X, S)$ be an infinite minimal subshift on the alphabet $\cA$ such that 
$K_1 = \liminf_{n\to +\infty} {p_X(n+1)-p_X(n)}$ is finite. 
Then, the topological rank of $(X,S)$ is bounded from above by
$(1+|\cA|^2K_1^2)^{2(K_1+2)}$. In particular, minimal subshifts with non-superlinear complexity have finite topological rank. 
\end{theo}
\begin{proof}
To get the  bound on  the  topological rank of the system $(X,S)$, thanks to \cref{lem:sufforfiniterank} it is enough to exhibit a sequence $(\cW _n)_{n \ge 1}$ of finite sets of words in $\cA^*$ such that  $\sup_{n\geq 1} |\cW_n| \le  (1+|\cA|^2K_1^2)^{K_1+2}$, $\min_{w\in \cW_n} |w|$  goes to infinity with $n$ and each $\cW_n$ is recognizable in $X$.  


Let $(m_i)_{i\geq 1}$ be an increasing sequence of positive integers such that $p_X(m_i+1)-p_X(m_i)$ converges to $K_1$ when $i$ goes to infinity.
It follows that there are at most $K_1$ right special words of length $m_i$ for every $i\geq 1$. 

Fix $i\geq 1$. Let $\mathcal{RS}_{m_i}(X)$ be the set of right special words of length $m_i$ in $\cL(X)$. Since any first return word between $w_1$ and $w_2$ in $\mathcal{RS}_{m_i}(X)$ corresponds to a path in the Rauzy graph $G_{m_i}$ from $w_1$ to $w_2$ in the complementary of $\mathcal{RS}_{m_i}(X)$ except the border, by \cref{lem:Monteil} the cardinality of the set of first return words between all the pairs of special words $\cR(\mathcal{RS}_{m_i}(X))$ is bounded from above by $|\cA|^2K_1^2$.

Let $\ell_i$ be the largest positive integer such that any two occurrences of the same special word of length $m_i$ are at least $\ell_i$ appart. Then, a standard compactness argument gives that $\ell_i$ goes to infinity as $i$ goes to infinity, otherwise we would have a periodic point.  

For $n\in \N$ write $\cR(\mathcal{RS}_{m_i}(X))=\cR_i^{\geq n}\cup \cR_i^{<n}$, where $\cR_i^{\geq n }$ is the set of all words in $\cR(\mathcal{RS}_{m_i}(X))$ of length larger than or equal to $n$ and $\cR_i^{<n}$ contains the words of length smaller than $n$. 
We claim that any concatenation of $K_1+1$  words in $\cR(\mathcal{RS}_{m_i}(X))$ occurring in a  $\cR(\mathcal{RS}_{m_i}(X))$-factorization has a word in $\cR_i^{\geq \lfloor \ell_i/(K_1+1) \rfloor}$.
To prove this statement, for $x \in X$ let $(p_j)_{j\in \Z}$ be the consecutive occurrences of the words of $\mathcal{RS}_{m_i}(X)$ in $x$, {\it i.e.}, $\{p_j : j\in \Z \} = \{ \ell \in \Z : x_{[\ell, \ell+m_i)} \in \mathcal{RS}_{m_i}(X)  \}$. For each index $j \in \Z$ set    
$u_j = x_{[p_j, p_{j+1})} \in \cR(\mathcal{RS}_{m_i}(X))$.
We have, 
\[
(K_1 +1)\max_{j\le \ell \le j+ K_1} |u_j| \ge  |u_j\cdots u_{j+ K_1 }| \ge \ell_i.
\]
This proves that for any $j\in \Z$ there 
is  $\ell \in \{j, \ldots,j+  K_1 \}$
such that $u_\ell$ belongs to $\cR_i^{\geq \lfloor \ell_i/(K_1+1) \rfloor }$, thus our claim holds.


For $n\geq 1$ let $i=i(n) \geq 1$ be large enough so that $\ell_{i(n)} \ge (K_1+1)n$ and thus, by previous claim, $\cR_{i(n)}^{\geq n}\neq \emptyset$. 
Let $(Y_i,S) =(Y_{\cR( \mathcal{RS}_{m_i}(X))}, S)$ be the subshift of return words to $\mathcal{RS}_{m_i}(X)$  in $X$ and $\tau_i$ an associated coding morphism. For convenience, denote by $\cB_i$ the alphabet of $Y_i$, so $\tau_i$ is defined from $ \cB_i^*$ to $\cR(\mathcal{RS}_{m_i}(X))^*$. 
Let $Z_i$ be the return word subshift associated to the set (of letters) $\{a\in \cB_i: \tau_i(a) \in \cR_{i}^{\geq n}\}$ in $Y_i$. Let $\cC_i$ denote the alphabet of $Z_i$ and  $\sigma_i \colon \cC_i^* \to \cB_i^*$ denote an associated coding morphism. Finally we set $\cW_n = \tau_i \circ \sigma_i (\cC_i) \subset \cL(X)$. 

Observe that by previous claim any word  in $\cW_n$ is of the form $u_0 u_1\ldots u_{k}$ with $1 \le k \le K_1+1$, $u_{0} \in \cR_i^{\geq n}$ and $u_j\in \cR_i^{<n}\cup\{\epsilon\}$ for $0< j \leq k$.
It follows that  $|\cW_{n}|$ is bounded by $|\cR_i^{\geq n}|(1+|\cR_i^{<n}|)^{K_1}$, which is bounded by $|\cA|^2 K_1^2(1+|\cA|^2 K_1^2)^{K_1}$ uniformly in $i$ and  $n$.
Also, by definition, we have  $\min_{w\in \cW_n} |w|\geq n$. 
Moreover, \cref{lem:RecognizabilityReturnWords} implies that $\sigma_i$ and $\tau_i$ are recognizable in $Z_i$ and in $Y_i$ respectively. By \cref{lem:CompMorphRec}, $\tau_i \circ \sigma_i$ is recognizable in $Z_i$. Finally, \cref{lem:fact-rec}  gives  that $\cW_n$ is recognizable in $X$. This shows the result.
\end{proof}

\begin{Rem}\rm
In the previous proof, it is worth to notice that if $\min_{w\in  \cR(\mathcal{RS}_{m_i}(X))}|w|$ goes to infinity as $i$ goes to infinity, then for $n\in \N^*$ and large enough $i$, 
$\cR_i^{<n}$ is empty  and therefore the topological rank of $X$ can actually be bounded by $|\cR(\mathcal{RS}_{m_i}(X))| \le |\cA|^2K_1^2$.
\end{Rem}

\section{Complexities of finite topological rank $\cS$-adic subshifts}
\label{sec:ComplexityFR}

In previous section we provide conditions on the complexity function of a minimal subshift (namely of non-superlinear complexity) to ensure finite topological rank (see \cref{cor:NonSuperLinareFR}). 
This section is devoted to study (partially) the converse. That is, the complexity function of a finite topological rank subshift and more generally of a $\cS$-adic subshift.  
We introduce notions of relative complexity for the morphisms defining the subshift and derive upper bounds for the complexity function.
We recover the well known result that finite topological rank Cantor systems have zero entropy (see {\it e.g.} \cite{Boyle&Handelman:1994}) even if for an arbitrary subexponential function $g \colon \N \to \N$ there are examples of finite topological rank subshifts whose complexity function is not dominated by $g$ \cite{Donoso&Durand&Maass&Petite:2016}.
In addition, we give several sufficient conditions guaranteeing that a finite topological rank subshift has non-superlinear complexity and show that subshifts with alphabet rank equal to 2 have non-superquadratic complexity. 
Next we propose bounds for the complexity function in terms of combinatorial complexities. Thus we deduce a result related to topological orbit equivalence: any strong orbit equivalence class of a finite topological rank minimal subshift contains a subshift of sub-linear complexity (\cref{cor:SOEsubLinComp}).
Finally in Subsection \ref{sec:Rk2NonSperline} we provide a subshift of topological rank $2$ with superlinear complexity, showing that the converse of \cref{cor:NonSuperLinareFR} does not hold.

\subsection{Relative complexity of morphisms}
For a morphism $\tau\colon \cA^*\to \cB^{\ast}$ set  \[
\|\tau\| = \max_{a \in \cA} |\tau(a)| \textrm{ and  }   \langle \tau \rangle =\min_{a \in \cA} |\tau(a)|.
\]

\begin{defi}
Let $\sigma\colon \cB^{\ast}\to \cC^{\ast}$ be a morphism. For $b\in \cB$ and $\cL\subseteq B^{\ast}$, define the following subset of words that extends $b$ to the right. 
\[ \mathcal{F}(b,n,\sigma,\cL) =\left \{w\in \cL: bw\in \cL, w_1\neq b,  |\sigma(w_{[1,|w|-1]})|< n\leq |\sigma(w|  \right \}. \] 
\end{defi}
That is, each word in $\mathcal{F}(b,n,\sigma,\cL)$ extends to the right some appearance of $b$ in $\cL$,  starts with a letter different from $b$ and has the shortest image under $\sigma$ of length at least $n$.

\begin{defi} Let $\tau\colon\cA^{\ast}\to \cB^{\ast}$ and $\sigma\colon \cB^{\ast}\to \cC^{\ast}$ be two morphisms. For $i\in \N^{\ast}$, we define the  {\em $i$-th  complexity map of} $\sigma$ {\em with respect to} $\tau$
as
\[\compi{\sigma}{\tau}{n}{i}=\sum_{b\in \cB}|\mathcal{F}(b,n, \sigma, \tau(\cA^{i}))|. \]
The {\em complexity map of} $\sigma$ {\em with respect to} $\tau$
is defined by  \[ \comp{\sigma}{\tau}{n}=\sum_{b\in \cB} | \mathcal{F}(b,n, \sigma, \cL(\tau(\cA^{\Z})))|. \]
\end{defi}
Note that $\compi{\sigma}{\tau}{n}{i}\leq \compi{\sigma}{\tau}{n}{i+1} \leq \comp{\sigma}{\tau}{n}$ for all $i,n\in \N^{\ast}$. 
Also, observe  that  $(\comp{\sigma}{\tau}{n})_n$ is increasing, whereas $(\compi{\sigma}{\tau}{n}{i})_n$ is not. Actually it is zero for all large enough $n\in \N$. 
We summarize some elementary properties of $\comp{\sigma}{\tau}{\cdot}$ for later use. 
\begin{lem}  \label{lem:propertiesRelComplexity}
For morphisms $\tau\colon \cA^*\to \cB^{\ast}$ and $\sigma\colon \cB^*\to\cC^{\ast}$ the following properties hold.
\begin{enumerate}[ref=\cref{lem:propertiesRelComplexity}-(\arabic*)]
    \item \label{item:RelComp1} 
    $\comp{\sigma}{\tau}{\|\sigma\|}\leq |\cB|^{\frac{\|\sigma\|}{\langle \sigma \rangle} +1 } $.  
    
    \item \label{item:RelComp2} If $|\cB|=2$, then $\comp{\sigma}{\tau}{\|\sigma\|}\leq \frac{\|\sigma\|}{\langle \sigma \rangle} +1. $
    
\end{enumerate}
\end{lem}

\begin{proof}
(1) It suffices to observe that for any $b\in \cB$ there are at most $|\cB|^{\frac{\|\sigma\|}{\langle \sigma\rangle }}$ words $w$ in  $\mathcal{F}(b,\|\sigma\|, \sigma, \cL(\tau(\cA^{\Z})))$.  

(2) Write $\cB=\{b_1,b_2\}$ with $|\sigma(b_1)|=\|\sigma\|$ and $|\sigma(b_2)|=\langle \sigma\rangle$.
Note that if $w$ belongs tp $ \mathcal{F}(b_1,\| \sigma\|,\sigma, \cL(\tau(\cA^{\Z})))$, then $w$ contains at most one $b_1$ as a letter. Hence $|\mathcal{F}(b_1,\| \sigma\|,\sigma, \cL(\tau(\cA^{\Z})))|\leq \frac{\|\sigma \| }{\langle \sigma \rangle}$.
On the other hand, $\mathcal{F}(b_2,\| \sigma\|,\sigma, \cL(\tau(\cA^{\Z})))=\{b_1\}$.

We conclude that $\comp{\sigma}{\tau}{\| \sigma\| }\leq \frac{\|\sigma \| }{\langle \sigma \rangle} +1 $.
\end{proof}

\begin{prop} \label{prop:bound2morphisms}
Let $\tau\colon \cA^*\to \cB^{\ast}$ and $\sigma\colon \cB^*\to \cC^{\ast}$ be two morphisms and suppose that $\tau$ is positive. Then we have that 
\[ p_{\sigma\circ\tau(\cA^{\Z})}(n)\leq (|\cB|+ \compi{\sigma}{\tau}{n}{2}) \cdot n\] 
for all $n\in [\|\sigma\|,\langle \sigma\circ\tau\rangle]$.
\end{prop}

\begin{proof}  
By definition, any word in $\sigma\circ\tau(\cA^{\Z})$ of length $n\in [\|\sigma\|,\langle \sigma\circ\tau\rangle]$ is a subword of $\sigma(b_{i_1})^{l_1}\ldots \sigma(b_{i_k})^{l_k}$, where $b_{i_1}^{l_1}\ldots b_{i_k}^{l_k} \in \tau(\cA^{2})$, $b_{i_j}\in \cB$ and $l_{j}\in \N^{\ast}$. 
We may consider two disjoint types of words of length $n$. 
\begin{itemize}
    \item[\textbf{Type 1}:] Words of length $n$ that start and end in a same block $\sigma(b)^{l}$ for some $b\in \cB$ and $l\in \N$. For each $b\in \cB$, such words are no more than $|\sigma(b)|$ and therefore there are no more than $|\cB|\|\sigma\|$ words of this type.   
    
    \item[\textbf{Type 2}:] Words of length $n$ that start in a block $\sigma(b)^{l}$ and end in another one. Any such word is a subword of $\sigma(b^{k}w)$, where $k=\lfloor n/|\sigma(b)|\rfloor$ and $w\in \mathcal{F}(b,n,\sigma, \tau(\cA^{2}))$. 
  Noting that there are at most $n$ subwords of length $n$ in $\sigma(b^{k}w)$ starting in the block $\sigma(b^{k})$ and ending in $\sigma(w)$, we conclude that words of \textbf{Type 2} are no more than 
   \[ \sum_{b\in \cB}n|\mathcal{F}(b,\sigma,n,\tau(\cA^{2}))|=\compi{\sigma}{\tau}{n}{2}\cdot n \] 
    \end{itemize}

Combining the two types of words provides the result. 
\end{proof}

Let us illustrate our approach by the following corollary that, even if not stated, can be deduced from \cite{Boyle&Handelman:1994}.
A straightforward consequence is the folklore result asserting that finite topological rank systems have zero entropy.

\begin{cor}\label{cor:FRZeroentropy}
Let $(X,S)$ be a $\cS$-adic subshift generated by the positive directed sequence $\boldsymbol{\tau}=(\tau_n\colon \cA_{n+1}^*\to \cA_{n}^{\ast})_{n\geq 0}$. 
If $\liminf_{n\to + \infty } \frac{\log(|\cA_{n+1}|)}{\langle {\tau}_{[0,n]}\rangle} = 0$,
then $(X,S)$ has zero entropy, {\it i.e.}, $\lim_{n\to 1} \log (p_X(n))/n=0$.

In particular, if $(X,S)$ is a minimal Cantor system of finite topological rank, then it has zero entropy.
\end{cor}
\begin{proof}                                                                                     
Let $\boldsymbol{\tau}=(\tau_n\colon \cA_{n+1}^*\to \cA_{n}^{\ast})_{n\geq 0}$ be a positive directed sequence of morphisms generating  $(X,S)$. By \ref{item:RelComp1}  we have that  
$\comp{{\tau}_{[0,n]}}{\tau_{n+1}}{\|{\tau}_{[0,n]}\|}\leq |\cA_{n+1}|^{\frac{\|{\tau}_{[0,n]}\|}{\langle {\tau}_{[0,n]} \rangle }+1}$. 
 \cref{prop:bound2morphisms} applied to the morphisms $\tau_{[0,n]}$ and $\tau_{n+1}$ implies that 
 \[  \frac{\log(p_{{\tau}_{[0,n+1]}(\cA_{n+2}^{\Z})}(\|{\tau}_{[0,n]}\|)}{ \|{\tau}_{[0,n]}\|} \leq \frac{\log(2|\cA_{n+1}|\|{\tau}_{[0,n]}\|)  }{\|{\tau}_{[0,n]}\| } + \frac{\|{\tau}_{[0,n]}\|}{\|{\tau}_{[0,n]}\| \langle {\tau}_{[0,n]}\rangle}\log(|\cA_{n+1}|), \] 
 which allows to conclude the first part of the statement. 
 The second part can be deduced from Theorem \ref{theo:equivalenceSadicvsFiniteRank} and Theorem \ref{DowMaass}.
 \end{proof}

As another application, we deduce some facts about the complexity of $\cS$-adic subshifts with an alphabet rank equal to 2. 

\begin{cor}\label{cor:SadicNonSuperQuadra}
Let $(X,S)$ be the $\cS$-adic subshift generated by the positive directive sequence $\boldsymbol{\tau}=(\tau_{n}\colon \cA^*_{n+1}\to \cA^*_{n})_{n\ge 0}$. If $\liminf_{n\to + \infty }|\cA_{n}|\leq 2$, then $(X,S)$ is sub-quadratic along a subsequence, \textit{i.e.},
\[\liminf_{n\to +\infty} \frac{p_{X}(n)}{n^2} = 0. \]
\end{cor}

\begin{proof}
Similarly to the previous proof, for $n\in \N$ consider the morphisms $\sigma={\tau}_{[0,n]}$ and $\tau_{n+1}$ and apply \cref{prop:bound2morphisms} and \ref{item:RelComp2} to obtain, for infinitely many $n$,
\begin{align*}p_{X}(\|{{\tau}_{[0,n]}}\|) & \leq p_{{\tau}_{[0,n]}\circ \tau_{n+1}(\cA_{n+2}^\mathbb{Z})}(\|{{\tau}_{[0,n]}}\|), \\ & \leq (2+\compi{{\tau}_{[0,n]}}{\tau_{n+1}}{ \|{\tau}_{[0,n]}\|}{2})  \|{{\tau}_{[0,n]}}\|, \\ 
&\leq \left( \frac{\|{{\tau}_{[0,n]}}\| }{\langle {{\tau}_{[0,n]}} \rangle} +3 \right) \|{{\tau}_{[0,n]}}\|. \end{align*}

Dividing by $\|{{\tau}_{[0,n]}}\|^2$ and taking the liminf in $n$ give the conclusion. 
\end{proof}

Concerning the growth rate of the complexity function of an $\cS$-adic subshift of alphabet rank $2$, \cref{cor:FRZeroentropy} shows it is sub-exponential.  
Beyond this, examples in \cite[Section 4]{Donoso&Durand&Maass&Petite:2016} show that there are no more precise bounds for the growth rate of the complexity function. Actually, for any sub-exponential function $\varphi\colon \N\to \R$ (\textit{i.e.}, $\limsup\limits_{n\to +\infty}\frac{\varphi(n)}{\alpha^n}=0$ for any $\alpha>1$), there are examples of $\cS$-adic subshifts with alphabet rank equal to 2 with $\limsup\limits_{n\to +\infty} \frac{p_X(n)}{\varphi(n)}>0$.  

\subsection{$\cS$-adic subshifts with non-superlinear complexity}
We provide several criteria that ensure that a $\cS$-adic subshift has a non-superlinear complexity. 

\begin{cor}  \label{cor:proptowersimplynonsup}
Let $(X,S)$ be the $\cS$-adic subshift generated by the positive directive sequence $\boldsymbol{\tau}=(\tau_{n}\colon \cA^*_{n+1}\to \cA^*_{n})_{n\ge 0}$. Assume that $\boldsymbol{\tau}$ has the proportional towers property, \textit{i.e.},  $K=\liminf_{n\to +\infty} \frac{\|{\tau}_{[0,n]}\|}{\langle {\tau}_{[0,n]} \rangle}<+\infty$ and that $(|\cA_n|)_{n\geq 0}$ is bounded. 
Then we have that \[ \liminf_{n\to +\infty} \frac{p_{X}(n)}{n} < +\infty.  \] 
\end{cor}
\begin{proof}
Let $(n_i)_{i\in \N}$ be a sequence in $\N$ such that $\frac{\|{\tau}_{[0,n_i]}\|}{\langle {\tau}_{[0,n_i]} \rangle} \leq  K$ for all $i\in \N$. Then by \cref{prop:bound2morphisms} and \ref{item:RelComp1} we have that
\begin{align*}
p_{X}(\|{{\tau}_{[0,n_i]}}\|) & \leq p_{{\tau}_{[0,n_i]}\circ \tau_{n_i+1}(\cA_{n_i+2}^\mathbb{Z})}(\|{{\tau}_{[0,n_i]}}\|), \\
& \leq (| \cA_{n_i+1}|^{\frac{\| {\tau}_{[0,n_i]}\| }{ \langle {\tau}_{[0,n_i]} \rangle } +1 }  +|\cA_{n_i+1}|) \|{{\tau}_{[0,n_i]}}\|, \\
& \leq (\max_{i\in \N} |A_i| ^{K+1}+ \max_{i\in \N} |A_i|) \|{{\tau}_{[0,n_i]}}\|,
\end{align*}
from where the result follows. 
\end{proof}

In terms of clopen Kakutani-Rohlin partitions, as a direct corollary of \cref{cor:proptowersimplynonsup} we obtain. 
 \begin{cor}
 Let $(X,T)$ be an expansive  minimal Cantor system with a sequence of CKR partitions $(\mathcal{T}_n)_{n\in \N}$ satisfying (KR1)-(KR4), $\mathcal{T}_n$ having $d_n$ towers and heights $(h_n (k))_{1\leq k \leq d_n}$ for each $n \in \N$. 
 Suppose they satisfy $\sup_{n\in \N}d_{n} < +\infty $ and  $\liminf_{n\to +\infty} \sup_{1 \le k, k' \le d_{n}} h_{n}(k) / h_{n}(k') < +\infty$.
 Then the system $(X,T)$ is topologically conjugate to a minimal subshift with a non-superlinear complexity.     
 \end{cor}

The preceding corollary applies to finite topological rank subshifts that admit tower partitions where all the towers of the same level have the same height (this is called the {\em equal row sum} property, or ERS for short). The class of subshifts having the ERS property coincides with the one of Toeplitz subshifts \cite[Theorem 8]{gjtoeplitz}. 
Although finite topological rank Toeplitz subshifts have the ERS property, we do not know if they have a Bratteli-Vershik representation having simultaneously the ERS property and a bounded number of towers at each level.
So we ask if within the class of Toeplitz subshifts, to have finite topological rank and non-superlinear complexity are equivalent conditions.
 
The following lemma gives a condition to deduce the hypothesis (hence the conclusion) of \cref{cor:proptowersimplynonsup} by 
analyzing the adjacency matrices associated to the morphisms of the given directive sequence. 
Before stating it, we introduce and recall the following definitions. For a letter $a \in \cA$, $|w|_{a}$ denotes the number of occurrences of the letter $a$ in $w$, and $\ell(w)$ is the vector $(|w|_{a})_{a\in \cA}$.  
For a morphism $\tau \colon \cA^* \to \cB^{\ast}$, its {\em incidence matrix}  is $M_{\tau} = (|\tau(a)|_{b})_{b \in \cB, a\in \cA}$. 
The following formulas are classical. 
For any word $w \in \cA^{*}$, $ \ell(\tau(w)) = M_{\tau}\ell(w)$ and for two morphisms $\tau_{0} \colon \cA_{1}^*\to \cA_{0}^{*}$ and $\tau_{1} \colon \cA_{2}^* \to \cA_{1}^{*}$, $M_{\tau_{0}\circ \tau_{1}} = M_{\tau_{0}}M_{\tau_{1}}$. For a positive matrix $M = (m_{i,j})_{(i,j) \in I\times J }$, we set  $D(M) = \sup_{i\in I} \sup_{j,k\in J} \frac{m_{i,j}}{m_{i,k}}$.
Note that a  bounded sequence of  positive matrices $(M_{n})_{n\geq 0}$ leads to a bounded sequence $(D(M_{n}))_{n\geq 0}$. However the converse is not true, for instance consider the matrices $M_{n}=\begin{pmatrix}
n & n+1 \\
n^{2} & (n+1)^{2}
\end{pmatrix}$. 

\begin{lem}\label{lem:SadicNonsuperlinear} 
Let $\boldsymbol{\tau} = (\tau_n  \colon  \mathcal{A}_{n+1}^* \to \mathcal{A}_n^*)_{n\ge 0}$ be a sequence of positive and proper morphisms.
Then, for any integer $n \ge 0$ we have
\[
\frac{\|  \boldsymbol{\tau}_{[0,n]}\|}{\langle  \boldsymbol{\tau}_{[0,n]} \rangle}
\le D(M_n).
\] 
In particular, $\liminf_{n\to +\infty} \frac{\| \boldsymbol{\tau}_{[0,n]}\|}{\langle  \boldsymbol{\tau}_{[0,n]}\rangle}$ is finite whenever $\liminf_{n\to +\infty} D(M_n)$ is finite.
\end{lem} 
 \begin{proof}
 Let $M_{n}$ denotes the incidence matrix of the morphism $\tau_{n}$. 
For  any letter $a \in \cA_{n+1}$, notice that  
\begin{eqnarray*}
| \tau_{0} \circ \cdots \circ \tau_{n} (a)| &= \langle (1,\ldots, 1)^{t} , \ell( \tau_{0} \circ \cdots \circ \tau_{n} (a)) \rangle , \\    
&= \langle (M_{0}\cdots M_{n-1})^{t} (1,\ldots, 1)^{t}, M_{n}\ell(a) \rangle,
\end{eqnarray*}
where $\langle \cdot, \cdot \rangle$ denotes the usual scalar product. 
Let $\{e_{c} : c\in \cA_{n} \}$ be the canonical basis and let $v_{c} >0$, $c\in \cA_{n}$, be a scalars such that $ (M_{0}\cdots M_{n-1})^{t} (1,\ldots, 1)^{t} =\sum_{c\in \cA_{n}} v_{c} e_{c}$. It follows that,
\[
  \langle (M_{0}\cdots M_{n-1})^{t} (1,\ldots, 1)^{t}, M_{n}\ell(a) \rangle = \sum_{c\in \cA_{n}} v_{c} \langle e_{c}, M_{n}\ell(a) \rangle.
\] 
Since $\langle e_{c}, M_{n}\ell(a) \rangle = | \tau_{n}(a)|_{c}$, we get for any pair of letters $a,b$ 
\[
| \tau_{0} \circ \cdots \circ \tau_{n} (a)| \le | \tau_{0} \circ \cdots \circ \tau_{n} (b)| \sup_{c\in \cA_{n}} \sup_{i,j\in \cA_{n+1}} \frac{|\tau_{n} (i) |_{c}}{|\tau_{n} (j) |_{c}}.
\]
This shows the lemma.
\end{proof}

This lemma applies for instance to $\cS$-adic subshifts given by morphisms  $(\tau_{n})_{n \ge 0}$ defined on alphabets of bounded size and such that $D(M_{n})$ is bounded for infinitely many $n\in \N$.

\subsection{The repetition complexity}
The relative complexities $\compi{\sigma}{\tau}{n}{i}$ may not be easy to compute in explicit examples. We introduce a related notion that is easier to analyze and that provides an upper bound for $\compi{\sigma}{\tau}{n}{i}$ for values of $n$ of interest. 

Let $\tau\colon \cA^*\to \cB^{\ast}$ be a morphism (we set $\cB=\{b_1,\ldots,b_{|\cB|}\}$). For $a\in \cA$, write $\tau(a)=b_{i_1}^{l_1}b_{i_2}^{l_2}\cdots b_{i_{k(a)}}^{l_{k(a)}}$  for some $i_1,\ldots,i_{k(a)}\in \{1,\ldots,|\cB|\}$, where $b_{i_{j+1}}\neq b_{i_j}$ for all $j=1,\ldots,k(a)-1$. The integer $k(a)$ represents the number of times one needs to switch letters while writing $\tau(a)$, plus one.  
The repetition complexity of $\tau$ is defined as
\[\rcomp (\tau)= \sum_{a\in \cA} k(a). \] 

Note that if $\tau$ is positive then $k(a)\geq |\cB|$ for all $a\in \cA$ and then $\rcomp (\tau)\geq |\cA||\cB|$.

\subsubsection*{Example:} We say that a morphism $\tau\colon \cA^*\to \cB^{\ast}$ is \textit{left to right} if there exists an ordering $b_1,\ldots,b_{|\cB|}$ of the elements of $\cB$ such that 
$\tau(a)=b_1^{l_1(a)}b_2^{l_2(a)}\cdots b_{|\cB|}^{l_{|\cB|}(a)}$ for all $a\in \cA$ for some $l_1(a),\ldots,l_{|\cB|}(a)\in \N^{\ast}$.
It is clear that if $\tau$ is left to right then $\rcomp (\tau)=|\cA||\cB|$ and so a left to right morphism attains the minimum value of the repetition complexity for primitive morphisms. 

\begin{lem}\label{lem:boundCompRep}
Let $\tau\colon \cA^*\to\cB^{\ast}$ and $\sigma\colon \cB^*\to\cC^{\ast}$ be two morphisms. The following properties hold:
\begin{enumerate} 
\item 
$\compi{\sigma}{\tau}{n}{1}\leq \rcomp (\tau)$, \text{ for all } $n\in \N$;
\label{item:P1Comp}
\item  $ \compi{\sigma}{\tau}{n}{2}\leq |\cA|\ \rcomp(\tau) +\compi{\sigma}{\tau}{n}{1} \leq (|\cA|+1)\ \rcomp (\tau)$,  \text{ for all }$n\in \N$. \label{item:P2comp}

\end{enumerate} 
\end{lem}

\begin{proof}
By definition, for any word $w \in \mathcal{F}(b,n,\sigma,\tau(\cA))$, the word  $bw $ appears in some $\tau(a)$, $a\in \cA$. We can associate to $w$ an occurrence of $b$ corresponding to the first letter of $bw$  in $\tau(a)$. Clearly this association is injective and each such occurrence of $b$ contributes by 1 to the value of $k(a)$. 
So    \[\compi{\sigma}{\tau}{n}{1}=\sum_{b\in \cB}|\mathcal{F}(b,n,\sigma,\tau(\cA)|\leq \sum_{a\in \cA}k(a)=\rcomp (\tau).\]
This proves \eqref{item:P1Comp}. 

To prove \eqref{item:P2comp}, note that for any word $w \in \mathcal{F}(b,n,\sigma,\tau(\cA^2))$, we have that $bw$ is a subword of $\tau(\cA^{2})$. So either  $bw$ is a subword of $\tau(a)$ for some $a\in \cA$ or it starts in $\tau(a)$ and ends in $\tau(a')$ for some $a,a'\in \cA$. 
Putting together all the words $w$ in the first case gives at most $\compi{\sigma}{\tau}{n}{1}$ words. 
In the second case, given a block $b^l$ in $\tau(a)$, and given $a'\in \cA$, there is a unique word $w\in \mathcal{F}(b,\sigma,n,\tau(\cA^{2}))$ such that $bw$ is a subword of $\tau(aa')$ (here the $b$ is the last one of the block $b^l$). Hence, there are at most $\sum_{a\in \cA}k(a)|\cA|=|\cA|\ \rcomp (\tau)$ words $w$ in this case. 
Combining the two cases gives the desired bound. 
\end{proof} 

\begin{theo}
\label{theo:ThreeMorphisms}
Let $\tau \colon  \cA^* \to \cB^*$, $\sigma \colon  \cB^* \to \cC^*$ and $\phi\colon\cC^{\ast} \to \cD^{\ast}$ be three morphisms and assume that $\tau$ and $\sigma$ are positive. Then, we have that  
\[p_{\phi\circ \sigma\circ \tau (\cA^{\Z}) } (n) \leq 
\begin{cases} (\compi{\phi}{\sigma}{n}{2}+|\cC|) n  & \text{ if }  n\in [\|\phi\|,\langle\phi\circ \sigma \rangle )\\(\compi{\phi\circ\sigma}{\tau}{n}{2}+\compi{\phi}{\sigma}{n}{1}+|\cB|+|\cC|) n  & \text{ if } 
n\in [\langle\phi\circ \sigma \rangle, \| \phi\circ \sigma\|]  \end{cases} \]
\end{theo}
\begin{proof} 
Writing $\cB=\{b_1,\ldots,b_{|\cB|}\}$ so that $|\sigma \circ \tau(b_i)|\leq| \sigma\circ\tau (b_{i+1})|$ for $i=1,\ldots,|\cB|-1$, one obtains
\[ 
[\|\phi \|, \| \phi \circ \sigma\|  ) = [\| \phi\|, |\phi \circ \sigma(b_1)|)  \cup \bigcup_{j=1}^{|\cB|-1} [|\phi \circ \sigma(b_j)|, |\phi \circ \sigma(b_{j+1})| ).\]
Let us analyze the number of words of length $n$ depending on the interval of the right hand side where $n$ belongs to. 

Let $n$ be  in  $[\|\phi\| , | \phi\circ \sigma (b_1) |)$. \cref{prop:bound2morphisms} gives 
\[p_{\phi\circ\sigma\circ\tau(\cA^{\Z})} (n) \leq p_{\phi\circ\sigma(\cB^{\Z})} (n) 
\leq ( |\cC|+\compi{\phi}{\sigma}{n}{2})n .\]
Now let $n$ be in  $ [|\phi \circ \sigma  (b_i)|, |\phi \circ \sigma  (b_{i+1})|)$ for $1 \le i < |\cB|$. 
Note that a word of length $n$ is a subword of $\phi \circ \sigma(b_{i_1}^{l_1}\cdots b_{i_k}^{l_k})$, where $b_{i_1}^{l_1}\cdots b_{i_k}^{l_k}$ belongs to $\tau(\cA^{2})$. 
We consider disjoint types of words of length $n$: 

\begin{itemize}
    \item[{\bf Type 1:}] Words of length $n$ that start and end in a block of the form $\phi \circ \sigma (b)^{l}$ and end in another block $\phi \circ \sigma \circ (b')^{l'}$. By the same argument as in \cref{prop:bound2morphisms}, there are no more than $\compi{\phi\circ\sigma}{\tau}{n}{2}\cdot n$ of such words.

\item[{\bf Type 2:}] Words of length $n$ starting and ending in a same block $(\phi \circ \sigma(b_j))^l$ for $b_j\in \cB$, where $b^l\in \tau(\cA^{\Z})$.  We may further distinguish two disjoint types: 

\begin{enumerate}
    \item[\textbf{Type 2.1:}] $j \le  i $ or equivalently   $n \ge |\phi\circ\sigma(b_{j})|$. In this case the number of words is at most $|\phi\circ \sigma(b_{j})|$. Moreover, such words are the same for all the different concatenations of $\phi\circ \sigma (b_j)$ among all the words $\phi \circ \sigma \circ \tau  (a)$. It follows there are at most  $\sum_{b: |\phi\circ\sigma(b)|\leq n } |\phi\circ \sigma (b)| \le in \le |\cB| n $ of such words. 

 \item[\textbf{Type 2.2:}] $j >  i $ or equivalently $n<|\phi\circ\sigma(b_{j})|$.  
In this case we may write $\sigma (b_{j}) = c_1^{l_1} \cdots c_{e(b_j)}^{l_{e(b_j)}} \in \sigma(\cB)$ and then
\[\phi\circ\sigma(b_{j})= \phi(c_1)^{l_1} \cdots \phi(c_{e(b_j)})^{l_{e(b_j)}}. \]
We will use the same distinction of cases as before to bound  the number of words, in this context, satisfying {\bf Type 2.2}: 

\begin{itemize} 
\item[\textbf{Type 2.2.1:}]
The number of  words in {\bf Type 2.2}  that start in some block $\phi(c_k)^{l_k}$ and end in another one is bounded by $n$. There are at most $\compi{\phi}{\sigma}{n}{1} \cdot n$ words in this case.

\item[\textbf{Type 2.2.2:}] We bound the number of words starting and ending in a same block $\phi(c_k)^{l_k}$. 
As $|\phi\circ\sigma(b_{j})| > n\geq  \|\phi \|$, there are at most $\| \phi \|$ of such words, whatever the power $l_k$ is. Hence among all the letters $c_k$ and $b_j$ the number of such words is bounded by  $|\cC| \| \phi \| \le |\cC|n. $ 
\end{itemize} 
\end{enumerate}
\end{itemize}
Summing up all the bounds obtained above, we obtain the result. 
\end{proof}

Combining \cref{theo:ThreeMorphisms} and \cref{lem:boundCompRep} we get. 

\begin{cor} \label{coro:threerepetitivebound} 
Let $\tau \colon  \cA^* \to \cB^*$, $\sigma \colon  \cB^* \to \cC^*$ and $\phi\colon\cC^{\ast} \to \cD^{\ast}$ be three morphisms and assume that $\tau$ and $\sigma$ are positive. 
Then, we have that 
\[p_{\phi\circ \sigma\circ \tau (\cA^{\Z}) } (n) \leq (\max\{ |\cA|,|\cB|,|\cC|\}+2)(\max\{\rcomp (\tau), \rcomp (\sigma)\}+1)\cdot n   \]
for all  $n \in [\| \phi \| , \|\phi \circ \sigma\|)$.
\end{cor}

As a direct consequence, we obtain the following. 

\begin{cor}\label{coro:threerepetitivebound2}
Let $(X,S)$ be the $\cS$-adic subshift generated by the positive directive sequence $\boldsymbol{\tau}= (\tau_n \colon  \mathcal{A}_{n+1}^*\to \mathcal{A}_n^*)_{n\ge 0}$.
Then, we have that 
 \[
p_X (n) \leq   
 (\max_{i\in \N} |\cA_{i}|+2)(\limsup_{i\to +\infty}{\rcomp (\tau_i)}+1 )\cdot n  
\]
for all large enough $n\in \N$.
\end{cor}

\begin{proof} Let $K\in \N$ be such that $\rcomp (\tau_n)\leq \limsup_{i\to + \infty}{\rcomp (\tau_i)}$ for all $n\geq K$.
For $n\in \N$, let $M\in \N$ be such that $n\in [\|{\tau}_{[0,M)}  \|, \| {\tau}_{[0,M+1)}\|]$. 
If $n$ is large enough, then as $M$ is greater than $K$,   \cref{coro:threerepetitivebound} applied  to $\phi = {\tau}_{[0,M)}$, $\sigma = \tau_{M}$ and $\tau=\tau_{M+1}$ provides

\begin{align*}p_X(n) &\leq   (\max_{i\in \N} |\cA_{i}|+2)(\limsup_{i\to + \infty} \rcomp (\tau_i) +1 )\cdot n
\end{align*} 
which finishes the proof. 
\end{proof}

The next result is an application of the bounds obtained above to topological orbit equivalence theory. 

\begin{cor}\label{cor:SOEsubLinComp}
Let $(X,S)$ be a finite topological rank minimal Cantor system. Then $(X,S)$ is strongly orbit equivalent to a subshift of sub-linear complexity.
\end{cor}

\begin{proof}
We assume $X$ is infinite, the finite case follows directly.  Recall from \cite{Giordano&Putnam&Skau:1995} that minimal Cantor systems are strong orbit equivalent whenever they can be represented using the same (unordered) Bratteli diagram (up to telescoping).  
Let $(V,E\preceq )$ be a Bratteli-Vershik representation of $(X,S)$. 
Let us consider the (unordered) Bratteli diagram $B=(V,E)$.
After telescoping and microscoping, we may assume that $|V_n|\geq 3$ for all $n\geq 1$.
We assign an ordering to $B$.
To do this we consider morphisms $\tau_n\colon V_{n+1}^{\ast}\to V_{n}^{\ast}$, $n\geq 1$, of the form 
$$v \mapsto w_1^{l_1(v)}w_{2}^{l_2(v)}w_1^{l_1'(v)}w_2^{l_2'(v)}w_3^{l_3(v)}\cdots w_{|V_n|}^{l_{|V_n|}},$$ 
for all  $v\in V_{n+1}$, where $w_1,\ldots,w_{|V_n|}$ is an ordering of $V_{n}$ (and for $n=0$ just consider a left to right morphism $\tau_0\colon V_1^{\ast}\to E_1^{\ast}$). It is proved in  \cite[Lemma 10]{gjtoeplitz} that it is possible to adjust the lengths $l_1,l_1',l_2,l_2'$ to obtain an expansive system (there it is also required that the minimum number of edges connecting a vertex of level $n$ with a vertex of level $n+1$ is at least $6|V_n|$, which can always be assumed after an appropriate telescoping). It is clear that the morphisms above are proper, recognizable and positive. \cref{prop:DurandLeroy2.3} tells us that the Bratteli-Vershik system associated to these morphisms is topologically conjugate to the corresponding $\cS$-adic subshift.
It is straightforward to compute that the repetition complexity of $\tau_{n}$ is equal to $|V_{n+1}|(|V_{n}|+2)$ (which is bounded in $n$). 
Finally, \cref{coro:threerepetitivebound2} allows us to conclude  that the complexity of the $\cS$-adic subshift generated by $(\tau_n)_{n\geq 0}$ is sub-linear. 
\end{proof}

\subsection{Example: a rank $2$ system with a superlinear complexity}\label{sec:Rk2NonSperline}
In this section we construct a subshift $(X,S)$ such that $\liminf_{n\to +\infty}  {p_X (n)}/{n} = +\infty$ and whose topological rank is $2$.
Hence we show that finite topological rank does not imply the non-superlinear complexity property. Recall that by \cref{cor:NonSuperLinareFR} non-superlinear complexity does imply finite topological rank. 

We fix some increasing sequence of integers $(a_n)_{n\geq 0}$ (specified  later) and we consider the sequence of morphisms $\boldsymbol{\tau}=(\tau_{n}\colon\{ 0,1\}^* \to \{ 0,1 \}^*)_{n\ge 0}$ defined by 
$$
\tau_n (0) = 011 \hbox{ and } \tau_n (1) = 0^2 1 0^3 1  \cdots 0^{a_n} 1 .
$$
Let $(X_{\boldsymbol{\tau}}, S)$ be the $\cS$-adic subshift generated by $\boldsymbol{\tau}$.
Set $A_n =  | {\tau}_{[0,n]} (1) |$, $B_n = | {\tau}_{[0,n]} (0) |$ and $C_n = B_{n-1} +\cdots + B_1 + 1$.

To show that $\liminf_{n\to +\infty} {p_{X_{\boldsymbol{\tau}}} (n)}/{n} = +\infty $ we indeed show the stronger condition $\lim_{n\to +\infty} p_{X_{\boldsymbol{\tau}}} (n+1) - p_{X_{\boldsymbol{\tau}}} (n) = +\infty $. 

Recall that $X_{\boldsymbol{\tau}}^{(n)}$ denotes the subshift generated by $( \tau_k )_{k\geq n}$.
We state some useful observations in the following lemmas. The details are left to the interested reader.

\begin{lem}
\label{lemme:gammanrightspecial}
For all $n\geq 1$,  the set of prefixes of length $C_n+1$ of the words ${\tau}_{[0,n]} (0)$ and ${\tau}_{[0,n]} (1)$ is $\{{\tau}_{[0,n-1]} (0) {\tau}_{[0,n-2]} (0) \cdots \tau_0 (0)0a : \  a\in \{ 0,1\} \}$.
\end{lem}
\begin{proof}
It follows by induction after noting that $00$ and $01$ belong to $X_{\boldsymbol{\tau}}^{(n+1)}$. 
\end{proof}
\begin{lem}
\label{lemme:onetoone}
The directive sequence $\boldsymbol{\tau}$ is recognizable. 
\end{lem}
\begin{proof}
The proof is left to the reader.
\end{proof}

\begin{lem}
\label{lemme:bispecial}
For all $n\geq 1$ and $k \in [1 , a_{n+1}-1]$, the words 
$$
W_n (k) = {\tau}_{[0,n]} (0^{k-1}10^k){\tau}_{[0,n-1]} (0) \cdots \tau_0 (0)0
$$ 
are right special words  in $X_{\boldsymbol{\tau}}^{(n)}$.  
\end{lem}
\begin{proof}
For all $n\geq 1$ and $k \in [1 , a_{n+1}-1]$,
the words $0^{k-1}10^k0$ and $0^{k-1}10^k1$ belong to $\cL(X_{\boldsymbol{\tau}}^{(n+1)})$. 
Hence, ${\tau}_{[0,n]} (0^{k-1}10^k0)$ and ${\tau}_{[0,n]} (0^{k-1}10^k1)$ belong to $\cL(X_{\boldsymbol{\tau}}^{(n)})$.
\cref{lemme:gammanrightspecial} allows us to conclude.
\end{proof}

\begin{lem}
\label{lemma:atleast}
For all $n\geq 1 $, $k \in [1 , a_{n+1}-1]$ and $N\in [k B_n +C_n , (k+1)B_n +C_n)$ there are at least $\min (k , {A_n}/{2B_n})$ different right special words of length $N$.
\end{lem}
\begin{proof}
Let $n\geq 1 $, $k \in [1 , a_{n+1}-1]$ and $N\in [k B_n +C_n , (k+1)B_n +C_n)$.
For each $i\in [k+1 - \min (k , {A_n}/{2B_n}), k]$ let  $w_i$ be the suffix of length $N$ of $W_n (i)$.
It is well-defined as the length of $W_n (i)$ is greater than $(k+1)B_n + C_n$.
One has, 
\begin{align}
\label{eq:suffix}
w_i = p_i {\tau}_{[0,n]} (0^i){\tau}_{[0,n-1]} (0) \cdots \tau_1 (0)0,
\end{align}
where $p_i$ is a suffix of ${\tau}_{[0,n]} (0^{i-1}1)$.
From \cref{lemme:bispecial} these words are right special. 
It remains to show they are distinct.
Let $i$ and $j$ be distinct integers belonging to the interval $[k+1 - \min (k , {A_n}/{2B_n}) , k ]$.
As ${\tau}_{[0,n]} (0)$ has a length smaller than the one of ${\tau}_{[0,n]} (1)$ and is not a suffix of it, it is easy to check that $w_i$ and $w_j$ are distinct (see \eqref{eq:suffix}).
\end{proof}

We have all the ingredients to construct a subshift of rank 2 with superlinear but subquadratic complexity. 
\begin{prop}
There exists a minimal subshift $(X,S)$ of topological rank $2$ such that $\liminf_{n\to +\infty} {p_X(n)}/{n} = +\infty$ and $\lim_{n\to +\infty} p_X(n)/n^2=0$. 
\end{prop}
\begin{proof}
Let $(X_{\boldsymbol{\tau}},S)$ be the $\cS$-adic subshift generated by the previous morphisms $\boldsymbol{\tau}=(\tau_n\colon \{0,1\}^{\ast}\to \{0,1\}^{\ast})_{n\geq 0}$, where we choose an increasing integer valued sequence $(a_n)_{n\geq 0}$ satisfying $a_{n+1} >  (3n+4){A_n}/{B_n}$. 

The morphisms  $\tau_n$, $n\ge 0$, being positive,  proper and defining a recognizable  directive sequence (\cref{lemme:onetoone}), generate a rank 2 minimal subshift $(X_{\boldsymbol{\tau}},S)$, thanks to \cref{prop:DurandLeroy2.3}.

To show that $\liminf_{n\to +\infty} {p_X(n)}/{n} = +\infty $, it is enough  to prove that the number of right special factors goes to infinity. 
As $\lim_{n\to +\infty} \min (n , {A_n}/{2B_n})  =+\infty$,
from \cref{lemma:atleast} it suffices to show that the intervals $[k B_n +C_n , (k+1)B_n +C_n)$ with $n\leq k < a_{n+1}$, $n\geq 0$, cover a half line $[N, +\infty )$ for some $N$.
Thus we only need to check  that  $(n+1)B_{n+1} +C_{n+1}$ is less than $a_{n+1} B_n+C_n$:
\begin{align*}
(n+1)B_{n+1} +C_{n+1} & \leq (n+2)B_n + 2(n+1)A_n  +C_n \leq (3n+4)A_n +C_n \\
&  < a_{n+1} B_n + C_n  .
\end{align*} 

This proves that $p_X$ is superlinear.

For the subquadratic property we claim the following.  

\textbf{Claim:} For all $m,n\geq 0$ we have that $\compi{\tau_{[0,m]}}{\tau_{m+1}}{n}{2}$ is bounded by $6{n}/{B_m}$. 

To see this, note that a word in $w$ in $\mathcal{F}(\ell,\tau_{[0,m]},n,\tau_{m+1}(\cA_{m+2}^2))$, $\ell=0,1$, is either a subword of $\tau_{m+1}(a)$, $a\in \{0,1\}$, or it starts in $\tau_{m+1}(a)$ and ends in $\tau_{m+1}(a')$ for $a,a'\in \{0,1\}$.  

For the first type, any word $w$ in $\mathcal{F}(1,\tau_{[0,m]},n,\tau_{m+1}(\cA_{m+2}))$ is completely determined by its first occurrence of 1. Similarly for $\ell=0$, any word $w$ in the set $\mathcal{F}(0,\tau_{[0,m]},n,\tau_{m+1}(\cA_{m+2}))$ is completely determined by its second occurrence of 1. 
Since $|\tau_{[0,m]}(0)|=B_m$, there are no more than ${n}/{B_m}$ words $w$ of this type in each case, for $\ell=0$ or $1$.  

For the second type, for any $a,a'\in \{0,1\}$ there are at most ${n}/{B_m}$ words starting in $a$ and ending in $a'$, since $B_m=\langle\tau_{[0,m]}\rangle$. Hence there are at most $2^2{n}/{B_m}$ of such words. This gives the conclusion of the Claim.  

For $n\in \N^{\ast}$, let $M=M(n)\in \N$ such that $n\in [\|{\tau}_{[0,M)}  \|, \| {\tau}_{[0,M+1)}\|]$. 
   \cref{theo:ThreeMorphisms} applied  to $\phi = {\tau}_{[0,M)}$, $\sigma = \tau_{M}$ and $\tau=\tau_{M+1}$, together with the Claim, provide
\begin{align*}p_X(n) \leq \left(6\frac{n}{B_{M(n)}}+ 6\frac{n}{B_{M(n)-1}}+4\right)\cdot n. 
\end{align*}
Dividing both sides by $n^2$ enables us to conclude. 
\end{proof}

\section{Asymptotic components in systems of topological rank 2} \label{sec:RankTwo}
This section is devoted to highlight a rigidity property in minimal Cantor systems of topological rank 2. For these systems and minimal $\cS$-adic subshifts of alphabet rank equal to 2 we analyse a specific dynamical invariant, namely, their asymptotic components.

We show that minimal $\cS$-adic subshifts on two letters have at most two asymptotic components  and that a non-periodic minimal Cantor system of topological rank 2 has only one asymptotic component (\cref{cor:ACRank2}).
In particular, the  example in Section \ref{sec:Rk2NonSperline} shows it exist non-periodic minimal subshifts on two letters with only one asymptotic component but at the same time having superlinear complexity.  
A relevant application of these results concerns the automorphisms group of a subshift, {\it i.e.}, the group of self-homeomorphisms of the subshift commuting with the shift map.
A straightforward consequence of  \cite[Theorem 3.1]{Donoso&Durand&Maass&Petite:2016} for minimal subshifts of topological rank equal to 2 is that any automorphism is a power of the shift map.
The same result provides for a minimal $\cS$-adic subshift on two letters that the group generated by the shift map is a subgroup of index at most $2$ within the one generated by the automorphisms of the subshift. 
To be more complete in this topic, let us recall that using the strategy in \cite{Cyr&Kra:2015} and \cite[Lemma 5]{Quas&Zamboni:2004}, it is possible to show that non-superquadratic complexity subshifts are {\em coalescent}, {\it i.e.}, any endomorphism of the subshift is invertible (see \cite[Section 3]{Donoso&Durand&Maass&Petite:2017} for details). Thus \cref{cor:SadicNonSuperQuadra} provides that any $\cS$-adic subshift on two letters is coalescent.

We recall some classical terminology in topological dynamics.  

Let $(X,T)$ be a topological dynamical system, where $X$ is endowed with a metric $\rm dist(\cdot,\cdot)$. Two points $x, y \in X$ are said to be {\em (left)}\footnote{One can also consider {\em right} asymptotic pairs, but we do not need that here.} {\em asymptotic} whenever  $\lim_{n\to +\infty} \text{dist}(T^{-n} x, T^{-n} y) =0$. It is well known that an aperiodic subshift always admits (left) asymptotic points $x,y$ where $x\neq y$ (such pairs are also known as nontrivial asymptotic pairs), see \cite[Chapter 1]{Auslander:1988}. 
For simplicity, in what follows we use the word asymptotic to mean left asymptotic. 

Given $x,y \in X$ we say that the orbits ${\rm Orb}_T(x)$ and ${\rm Orb}_T(y)$ are  {\em asymptotic} if there exist points $x' \in {\rm Orb}_{T}(x)$ and $y' \in {\rm Orb}_{T}(y)$ that are asymptotic. This condition is equivalent to say that $y$ is  asymptotic to some $T^nx$ or vice versa.
Then, for each   $x' \in {\rm Orb}_{T}(x)$ there is a  point $y' \in {\rm Orb}_{T}(y)$ asymptotic to $x'$.
It is clear that this relation is an equivalence relation on the collection of orbits.
We call a non trivial  equivalence class, ({\it i.e.}, a class not reduced to a point) an {\em asymptotic component}.

We start with a lemma, which is a weaker version of the famous Fine and Wilf theorem \cite[Section 8]{Lothaire:2002}.

\begin{lem}\label{lem:prefix}
Let $A$ and $B$ be different non-empty words. Then, there exists a word $u$ of length strictly less than $|A|+|B|$ such that for any pair of different sequences $x=x_0 \cdots x_L\in A\{A,B\}^{*}$ and $y=y_0 \cdots y_M \in B\{A,B\}^{*}$, $u$ is a common prefix of $x$ and $y$ and, when  $|u|$ is not greater than $L$ and $M$, $x_{|u|} \neq y_{|u|}$.
\end{lem}
\begin{proof}
We prove the lemma by induction on the length   $||B|-|A||$. Assume first that $|A| = |B|$, so $||B|-|A||=0$. Since $A \neq B$, $u$ can be taken as the maximal common prefix of $A$ and $B$, which can be the empty word.

Let us assume the lemma is true for all words $A$ and $B$ such that $||B|-|A|| \leq n$ for all $n\geq 0$. 
Let $A$ and $B$ be two words such that  
$||B|-|A||=n+1$. Without loss of generality we may assume that $|A|<|B|$. If $A$ is not a prefix of $B$, take $u$ as the maximal common prefix of $A$ and $B$. Otherwise, $B$ has the form $AB'$ for some non-empty word $B'$. 

Let $x$ and $y$ be two sequences such that $x=x_0 \cdots x_L \in A\{A,B\}^{*}$ and $y=y_0 \cdots y_M \in B\{A,B\}^{*}$. 
Changing $B$ by $AB'$, we get that the sequence $x' =x_{[|A|,L]}$, if not empty, belongs to $A\{A,B'\}^{*}$ and $y'=y_{[|A|,M]}$ belongs to $B'\{A,B'\}^{*}$. The induction hypothesis applied to $A$ and $B'$ implies that it exists a word $u'$ of length strictly less than $|A|+|B'|=|B|$, independent of $x$ and $y$, which is a common prefix of $x'$ and $y'$ and 
$x_{|A|+|u'|} \neq y_{|A|+|u'|}$. The result follows by considering the word $u=Au'$. 
\end{proof}

We recall the notion of recognizability of a family of words with respect to a subshift as defined in Section \ref{sec:recog-factorizing}. 
Let $\cW\subset \cB^{\ast}$ be a set of non-empty words and $X\subseteq \cB^\Z$ be a subshift. We say $\cW$ is recognizable in $X$ if for a subshift $Y\subseteq \cA^\Z$ and a morphism $\tau:\cA^*\to \cB^*$ with $\tau(\cA)=\cW$ we have: $X$ is the smallest subshift containing $\tau(Y)$ and $\tau$ is recognizable in $Y$, ({\it i.e.}, every $x\in X$ has a unique $\cW$-factorization adapted to $(Y, \tau)$). 

The typical examples we will consider in this section are inspired by the following theorem that is a particular result of \cite{Berthe&Steiner&Thuswaldner&Yassawi:2018}.

\begin{theo}[{\cite[Theorem 4.6]{Berthe&Steiner&Thuswaldner&Yassawi:2018}}]\label{theo:BSTY}
Let $\cA$ be an alphabet with $2$ letters. 
Let $(X_{\boldsymbol{\tau}}, S)$ be an aperiodic $\cS$-adic subshift generated by the directive sequence
$\boldsymbol{\tau}=(\tau_{n}:\cA^*\to\cA^*)_{n\ge 0}$, with $\tau_{n}$ non-erasing for all $n\geq 0$.
Then
$\{\tau_{[0,n)}(a): a\in \cA\}$ is recognizable in $X_{\boldsymbol{\tau}}$ (with respect to $(X^{(n)}_{\boldsymbol{\tau}},\tau_{n-1})$)
for any $n \ge 1$.  
\end{theo}

The next proposition particularly applies to Cantor systems with topological rank equal to $2$ and to $\cS$-adic subshifts obtained from recognizable morphisms on alphabets of $2$ letters. 
For its proof we recall some notions and notations. Let $\cW \subset \cA^*$ be a set of non-empty words and $X\subseteq \cA^\Z$ be a subshift. Recall that the cutting points of a $\cW$-factorization $F_x \colon \Z \to \cW \times \Z$ of a point $x \in X$
is the set
$$C(F_x)= \{ i \in \Z: F_x(k) = (w,i) \textrm{ for some } k \in \Z, w \in \cW\}.$$   
We say the $\cW$-factorizations $F_x$ and $F_{x'}$ of $x$ and $x'$ in $X$ {\em coincide} on a set $J \subseteq \Z$ if $C(F_x) \cap J = C(F_{x'}) \cap J$. 

\begin{prop}\label{prop:ACRank2}
Let $(A_{n})_{n\ge 1}$ and $(B_{n})_{n\ge 1}$ be two sequences of non-empty words in $\cA^{*}$ with increasing length. Let $X \subseteq \cA^\Z$ be a subshift such that $\{A_{n},B_{n}\}$ is recognizable in $X$ for all $n \ge 1$. Then, $(X,S)$ admits at most $2$ asymptotic components.  
Moreover, if the maximal common suffix of $A_{n}$ and $B_{n}$ has a length going to infinity with $n$,
then $(X,S)$ admits at most $1$ asymptotic component.
\end{prop}
\begin{proof}  
By hypothesis, for all $n\geq 1$ there exist 
a morphism $\tau_n:\cA_n^* \to \cA^*$ with 
$\tau_n(\cA_n)=\cW_n=\{A_n,B_n\}$ and a subshift $Y_n\subseteq \cA_n^\Z$ such that 
$X$ is the smallest subshift containing $\tau_n(Y_n)$ and all $x\in X$ has a unique 
$\cW_n$-factorization adapted to $(Y_n,\tau_n)$. 
To avoid notations, in what follows the factorizations will be always adapted to the corresponding subshifts and morphisms, which will be clear from the context.

The proof is by contradiction. Let us assume that $\{x^{(0)},y^{(0)}\}$, $\{x^{(1)},y^{(1)}\}$, $\{ x^{(2)},y^{(2)}\}$ are left asymptotic pairs representatives of three different asymptotic components of $(X,S)$. 
After shifting them, we can assume that  $x^{(j)}_{(-\infty,-1]} = y^{(j)}_{(-\infty,-1]}$ and $x^{(j)}_{0} \neq y^{(j)}_{0} $ for each $j\in \{0,1,2\}$. 

Fix $n\in \N^{\ast}$ and $j\in \{0,1,2\}$. Since the distance between $S^{-k} x^{(j)}$ and $S^{-k} y^{(j)}$ goes to $0$ as $k$ goes to $+\infty$, then the fact that $\cW_n$ is recognizable in $X$ is a local property (see Lemma \ref{lem:rayon-rec}) implies that there exists an integer $\ell \le 0$ such that the $\cW_n$-factorizations of $x^{(j)}$ and $y^{(j)}$ coincide on the interval $(-\infty, \ell)\cap \Z$. Let $\ell^{(j)}$ be the largest such integer and denote by $F_j$ the $\cW_n$-factorization of $x^{(j)}$. 
We set $a^{(j)} = \max C(F_j)\cap (-\infty, \ell^{(j)})$.
In particular, $S^{a^{(j)}}x^{(j)}$ and $S^{a^{(j)}}y^{(j)}$ are in two different cylinders $[A_{n}]$, $[B_{n}]$.

Let $k_{j}$ be an integer and $w^{(j)}$ be a word in $\cW_n$ such that $F_j(k_j)=(w^{(j)}, a^{(j)}- |w^{(j)}|)$.
Then $w^{(j)}$ is the word $x^{(j)}_{[a^{(j)}- |w^{(j)}|,a^{(j)})}$ before the cutting point $a^{(j)}$. 

Let $u_{n}$ be the word given by \cref{lem:prefix} associated to $\cW_n=\{A_{n}, B_{n}\}$. The word $u_{n}$ is the maximal common prefix of $x^{(j)}_{[a^{(j)}, +\infty)}$ and $y^{(j)}_{[a^{(j)}, +\infty)}$ and is independent of $j$. 
So its length $|u_{n}|$ is equal to $|a^{(j)}|$ for any $j \in \{0,1,2\}$. By the pigeonhole principle, there are two indices $j_{1}$ and $j_{2} \in \{0,1,2\}$  sharing the same word $w^{(j_{1})}  = w^{(j_{2})}$. Hence the words $x^{(j)}_{[-|u_{n}|-|w^{(j)}|,0)}= w^{(j)}u_n$ are  all the same for $j \in \{j_{1}, j_{2}\}$. 
    
Repeating the same argument for each $n\geq 1$ we obtain a sequence of pairs of indices $\{j_{1}^{(n)}, j_{2}^{(n)} \}$ such that the word $x^{(j)}_{[-|u_{n}|-|w^{(j)}| 
,0)}$, with $ w^{(j)} $ in $ \{A_{n}, B_{n}\}$, is common  to each $j$ in $ \{j_{1}^{(n)}, j_{2}^{(n)} \}$. 
Again the pigeonhole principle ensures that there is a pair of two indices $ \{j_{1}, j_{2}\}$ common  to infinitely many integers $n$. Hence  $x^{(j_{1})}_{(-\infty, 0)}=x^{(j_{2})}_{(-\infty, 0)}$.
It follows that the sequences $x^{(j_{1})}$ and $x^{(j_{2})}$ are asymptotic, which is a contradiction.

In the case where the length $\ell_{n}$  of  the maximal common suffix of $A_{n}$ and $B_{n}$ goes to infinity with $n$, the words $x^{(j)}_{[-|u_{n}| - \ell_{n},0)}$  are all equal for any $j \in \{0,1,2\}$. Making $n$ converging to infinity allows to conclude that the sequences $x^{(0)}$, $x^{(1)}$ and $x^{(2)}$ are asymptotic, proving that in this case there exists at most one asymptotic component.
\end{proof} 

The next result is a direct application of \cref{prop:ACRank2} and \cref{theo:BSTY}. Recall that a topological rank $2$ minimal Cantor system is always conjugate to a $\cS$-adic subshift on two letters where the morphisms are proper (see \cref{section:finiterankvsSadic}).

\begin{cor}\label{cor:ACRank2}
Let $\cA$ be an alphabet with $2$ letters.
Suppose the $\cS$-adic subshift $(X_{\boldsymbol{\tau}}, S)$ generated by the sequence 
$\boldsymbol{\tau}=(\tau_{n}\colon \cA^* \to \cA^{\ast})_{n\ge 0}$ of non-erasing morphisms is aperiodic and  minimal. Then, $(X_{\boldsymbol{\tau}},S)$ has at most two asymptotic components. 

Let $(X,T)$ be a topological rank $2$ minimal Cantor system, then it has a unique asymptotic component.  
\end{cor}
As an example consider the  Prouhet-Thue-Morse  morphism (substitution) $\tau$ defined by $0 \mapsto 01$, $1\mapsto 10$. The subshift generated by this substitution is a $\cS$-adic subshift with only one substitution called {\em Prouhet-Thue-Morse subshift}.
The sequences $\lim_{n\to +\infty} \tau^{2n}(0).\lim_{n\to +\infty} \tau^{2n}(i)$ and  $\lim_{n\to +\infty} \tau^{2n}(1).\lim_{n\to +\infty} \tau^{2n}(i)$ are in the associated subshift and are  asymptotic for each $i \in \{0,1\}$. Hence, by Corollary \ref{cor:ACRank2}, the subshift has exactly two asymptotic components but it is not of topological rank $2$. Actually, using techniques of substitution theory on return words of \cite{Durand&Host&Skau:1999} and \cite{Durand&Leroy:2012}, a long but standard computation enables us to  create a $\cS$-adic subshift  conjugate to the Prouhet-Thue-Morse subshift, with a proper directive sequence of alphabet rank $3$.
Thus  its topological  rank  is not greater than $3$, so that its topological rank is exactly $3$.    

\section{Final remarks and questions}
\label{sec:questions.}

The bounds obtained on the topological rank (\cref{theo:equivalenceSadicvsFiniteRank},  \cref{lem:sufforfiniterank} or \cref{cor:NonSuperLinareFR})  for a minimal subshift  are far to be optimal. 
In addition, even for classical examples ({\it e.g.}, provided by substitutions), one observes that the computation of their topological rank is not an easy task. 
It makes us  wonder about the calculability of the topological rank of a minimal subshift or about the  alphabet rank of a $\cS$-adic subshift.
To be more precise, we may  consider this problem for systems described by explicit algorithms, like substitutive subshifts or effective $\cS$-adic subshifts. In this context, we ask
\begin{ques}
Do there exist algorithms that taking as input the data of a directive sequence of morphisms, compute the alphabet/topological rank of the corresponding $\cS$-adic subshift ? 
\end{ques} 
A less ambitious objective is to obtain lower and upper bounds for the alphabet/topological rank. In this respect  studying the complexity function of the subshift could provide partial answers. In light of  \cref{cor:SadicNonSuperQuadra} and \cref{cor:proptowersimplynonsup}, we ask:
\begin{ques}
For any $r\ge 1$ does there exists  $d >0 $ such that  any  aperiodic minimal $\cS$-adic subshift $(X,S)$  generated by a positive directive sequence $\boldsymbol{\tau}=(\tau_{n}\colon \cA^*_{n+1}\to \cA^*_{n})_{n\ge 0}$ with $\liminf |\cA_n| \le r$, satisfies  $\liminf_{n\to \infty}  p_X(n)/n^d=0$?
If not, provide a counter example.
\end{ques} 
The Prouhet-Thue-Morse subshift illustrates how the number of asymptotic components may provide a lower bound of the topological rank of a subshift. Moreover it is possible that results in \cref{sec:RankTwo} are available for  minimal subshifts of higher topological or alphabetical ranks.

\begin{ques} \label{ques:FiniteAsymptotic} 
How much of the analysis made in \cref{sec:RankTwo} can be extended to minimal systems of finite topological/alphabet rank  greater than two ?
In particular, do finite topological rank minimal Cantor systems have finitely many asymptotic components? 
\end{ques} 

This question is of independent interest, since a positive answer to it implies for instance that the automorphism groups of such systems are virtually $\Z$, by a direct application of results in \cite{Donoso&Durand&Maass&Petite:2016}. Right before finishing this article we got informed that a big progress has been made by B. Espinoza regarding \cref{ques:FiniteAsymptotic} (personal communication).  

In Theorem \ref{theo:equivalenceSadicvsFiniteRank} part (2) we stated that minimal $\cS$-adic subshifts of bounded alphabet rank are topological factors of finite topological rank minimal Cantor systems. It is natural to think that a $\cS$-adic subshift itself is of finite topological rank, but this relation is still unclear. So it is natural to ask in general, 

\begin{ques} 
Is a topological factor of a minimal Cantor system of finite topological rank of  finite topological rank too? 
\end{ques} 
We suspect this question has a positive answer since in the particular case of a subshift of non-superlinear complexity the subshifts factors also have a non-superlinear complexity, hence they have a finite topological rank.  

It is interesting (and maybe more easy to manage) to restrict all the former questions to a rich class of systems with a specific dynamical property, like Toeplitz subshifts. Already for this particular case the questions are open since we do not know a suitable characterization of Toeplitz subshifts with finite topological rank.
More precisely and motivated by the fact that any Toeplitz subshift admits a Bratteli-Vershik representation whose towers per level have the same height, we ask the following:
\begin{ques} \label{question:Toeplitz_Nonsuperlinear}
Let $(X,S)$ be a Toeplitz subshift. Is it true that 
$X$ has finite topological rank if and only if the complexity of $X$ is non-superlinear? \end{ques}
The subtle point of this question is that the representation of the  Toeplitz subshift with a bounded number of vertices at each level does not need to satisfy the ERS property (recall the discussion right after \cref{cor:proptowersimplynonsup}), and conversely, the representation satisfying the ERS property may not have a bounded number of vertices at each level.  To answer \cref{question:Toeplitz_Nonsuperlinear} it suffices to give a positive answer to the following: does a Toeplitz subshift of finite topological rank admit a representation satisfying the ERS property with a bounded number of vertices at each level? 

A positive answer to this question would allow us to better understand all previous questions within the family of Toeplitz subshifts.

\bibliography{biblio}

\newcommand{\etalchar}[1]{$^{#1}$}
\providecommand{\bysame}{\leavevmode\hbox to3em{\hrulefill}\thinspace}
\providecommand{\MR}{\relax\ifhmode\unskip\space\fi MR }
\providecommand{\MRhref}[2]{%
  \href{http://www.ams.org/mathscinet-getitem?mr=#1}{#2}
}
\providecommand{\href}[2]{#2}
\begin{thebibliography}{CDHM03}

\bibitem[Aus88]{Auslander:1988}
J.~Auslander, \emph{Minimal flows and their extensions}, North-Holland
  Mathematics Studies, vol. 153, North-Holland Publishing Co., Amsterdam, 1988,
  Notas de Matem{\'a}tica [Mathematical Notes], 122.

\bibitem[BCBD{\etalchar{+}}]{Berthe&Cecchi&Durand&Perrin&Petite:2019}
V.~Berth{\'e}, P.~Cecchi-Bernales, F.~Durand, D.~Perrin, and S.~Petite,
  \emph{On the dimension group of unimodular {S}-adic subshifts}, arXiv
  e-prints, arXiv:1911.07700.

\bibitem[BDM10]{Bressaud&Durand&Maass:2010}
X.~Bressaud, F.~Durand, and A.~Maass, \emph{On the eigenvalues of finite rank
  {B}ratteli-{V}ershik dynamical systems}, Ergodic Theory Dynam. Systems
  \textbf{30} (2010), 639--664.

\bibitem[BH94]{Boyle&Handelman:1994}
M.~Boyle and D.~Handelman, \emph{Entropy versus orbit equivalence for minimal
  homeomorphisms}, Pacific J. Math. \textbf{164} (1994), 1--13.

\bibitem[BKM08]{Bezuglyi&Kwiatkowski&Medynets:2008}
S.~Bezuglyi, J.~Kwiatkowski, and K.~Medynets, \emph{Aperiodic substitutional
  systems and their bratteli diagrams}, Ergodic Theory Dynam. Systems
  \textbf{29} (2008), 37--72.

\bibitem[BKMS13]{Bezuglyi&Kwiatkowski&Medynets&Solomyak:2013}
S.~Bezuglyi, J.~Kwiatkowski, K.~Medynets, and B.~Solomyak, \emph{Finite rank
  {B}ratteli diagrams: structure of invariant measures}, Trans. Amer. Math.
  Soc. \textbf{365} (2013), 2637--2679.

\bibitem[Bos85]{Boshernitzan:1984}
M.~Boshernitzan, \emph{A unique ergodicity of minimal symbolic flows with
  linear block growth}, J. Analyse Math. \textbf{44} (1984/85), 77--96.

\bibitem[BSTY19]{Berthe&Steiner&Thuswaldner&Yassawi:2018}
V.~Berth{\'e}, W.~Steiner, J.~Thuswaldner, and R.~Yassawi,
  \emph{Recognizability for sequences of morphisms}, Ergodic Theory Dynam.
  Systems \textbf{39} (2019), 2896–2931.

\bibitem[CDHM03]{Cortez&Durand&Host&Maass:2003}
M.~I. Cortez, F.~Durand, B.~Host, and A.~Maass, \emph{Continuous and measurable
  eigenfunctions of linearly recurrent dynamical cantor systems}, J. of the
  London Math. Soc. \textbf{67} (2003), no.~3, 790--804.

\bibitem[CDP16]{Cortez&Durand&Petite:2016}
M.~I. Cortez, F.~Durand, and S.~Petite, \emph{Eigenvalues and strong orbit
  equivalence}, Ergodic Theory Dynam. Systems \textbf{36} (2016), 2419--2440.

\bibitem[CK15]{Cyr&Kra:2015}
V.~Cyr and B.~Kra, \emph{The automorphism group of a shift of linear growth:
  beyond transitivity}, Forum Math. Sigma \textbf{3} (2015), e5, 27.

\bibitem[CK19]{Cyr&Kra:2019measures}
\bysame, \emph{Counting generic measures for a subshift of linear growth}, J.
  Eur. Math. Soc. (JEMS) \textbf{21} (2019), no.~2, 355--380.

\bibitem[DDMP16]{Donoso&Durand&Maass&Petite:2016}
S.~Donoso, F.~Durand, A.~Maass, and S.~Petite, \emph{On automorphism groups of
  low complexity subshifts}, Ergodic Theory Dynam. Systems \textbf{36} (2016),
  64--95.

\bibitem[DDMP17]{Donoso&Durand&Maass&Petite:2017}
\bysame, \emph{On automorphism groups of {T}oeplitz subshifts}, Discrete Anal.
  (2017), Paper No. 11, 19.

\bibitem[DFM15]{Durand&Frank&Maass:2015}
F.~Durand, A.~Frank, and A.~Maass, \emph{Eigenvalues of {T}oeplitz minimal
  systems of finite topological rank}, Ergodic Theory Dynam. Systems
  \textbf{35} (2015), 2499--2528.

\bibitem[DFM19]{Durand&Frank&Maass:2019}
\bysame, \emph{Eigenvalues of minimal {C}antor systems}, J. Eur. Math. Soc.
  (JEMS) \textbf{21} (2019), 727--775.

\bibitem[DHS99]{Durand&Host&Skau:1999}
F.~Durand, B.~Host, and C.~Skau, \emph{Substitutive dynamical systems,
  {B}ratteli diagrams and dimension groups}, Ergodic Theory Dynam. Systems
  \textbf{19} (1999), 953--993.

\bibitem[DL12]{Durand&Leroy:2012}
F.~Durand and J.~Leroy, \emph{{$S$}-adic conjecture and {B}ratteli diagrams},
  C. R. Math. Acad. Sci. Paris \textbf{350} (2012), 979--983.

\bibitem[DM08]{Downarowicz&Maass:2008}
T.~Downarowicz and A.~Maass, \emph{Finite-rank {B}ratteli-{V}ershik diagrams
  are expansive}, Ergodic Theory Dynam. Systems \textbf{28} (2008), 739--747.

\bibitem[Dur98]{Durand:1998}
F.~Durand, \emph{A characterization of substitutive sequences using return
  words}, Discrete Math. \textbf{179} (1998), no.~1-3, 89--101.

\bibitem[Dur00]{du1}
F.~Durand, \emph{{L}inearly recurrent subshifts have a finite number of
  non-periodic subshift factors}, Ergodic Theory Dynam. Systems \textbf{20}
  (2000), 1061--1078.

\bibitem[Dur10]{review}
\bysame, \emph{Combinatorics on {B}ratteli diagrams and dynamical systems},
  Combinatorics, automata and number theory, Encyclopedia Math. Appl., vol.
  135, Cambridge Univ. Press, Cambridge, 2010, pp.~324--372.

\bibitem[Fer96]{Ferenczi:1996}
S.~Ferenczi, \emph{Rank and symbolic complexity}, Ergodic Theory Dynam. Systems
  \textbf{16} (1996), 663--682.

\bibitem[FM10]{Ferenczi&Monteil:2010}
S.~Ferenczi and T.~Monteil, \emph{Infinite words with uniform frequencies, and
  invariant measures}, Combinatorics, automata and number theory, Encyclopedia
  Math. Appl., vol. 135, Cambridge Univ. Press, Cambridge, 2010, pp.~373--409.

\bibitem[GHH18]{Giordano&Handelman&Hosseini:2018}
T.~Giordano, D.~Handelman, and M.~Hosseini, \emph{Orbit equivalence of {C}antor
  minimal systems and their continuous spectra}, Math. Z. \textbf{289} (2018),
  1199--1218.

\bibitem[GJ00]{gjtoeplitz}
R.~Gjerde and {\O}.~Johansen, \emph{{B}ratteli-{V}ershik models for {C}antor
  minimal systems: applications to {T}oeplitz flows}, Ergodic Theory Dynam.
  Systems \textbf{20} (2000), 1687--1710.

\bibitem[GPS95]{Giordano&Putnam&Skau:1995}
T.~Giordano, I.~F. Putnam, and C.~F. Skau, \emph{Topological orbit equivalence
  and {$C^*$}-crossed products}, J. Reine Angew. Math. \textbf{469} (1995),
  51--111.

\bibitem[H{\o}y17]{Hoynes:2017}
S.-M. H{\o}ynes, \emph{Finite-rank {B}ratteli-{V}ershik diagrams are
  expansive---a new proof}, Math. Scand. \textbf{120} (2017), 195--210.

\bibitem[HPS92]{hps}
R.~Herman, I.~Putnam, and C.~Skau, \emph{Ordered {B}ratteli diagrams, dimension
  groups and topological dynamics}, Internat. J. Math. \textbf{3} (1992),
  827--864.

\bibitem[KM02]{Karhumaki&Manuch:2002}
J.~Karhum{\"{a}}ki and J.~Manuch, \emph{Multiple factorizations of words and
  defect effect}, Theoret. Comput. Sci. \textbf{273} (2002), 81--97, WORDS
  (Rouen, 1999).

\bibitem[Lot02]{Lothaire:2002}
M.~Lothaire, \emph{Algebraic combinatorics on words}, Encyclopedia of
  Mathematics and its Applications, vol.~90, Cambridge University Press,
  Cambridge, 2002.

\bibitem[Mos92]{Mosse:1992}
B.~Moss{\'e}, \emph{Puissances de mots et reconnaissabilit\'e des points fixes
  d'une substitution}, Theoret. Comput. Sci. \textbf{99} (1992), 327--334.

\bibitem[ORW82]{Ornstein&Rudolph&Weiss:1982}
D.~Ornstein, D.~Rudolph, and B.~Weiss, \emph{Equivalence of measure preserving
  transformations}, Mem. Amer. Math. Soc. \textbf{37} (1982), no.~262, xii+116.

\bibitem[QZ04]{Quas&Zamboni:2004}
A.~Quas and L.~Zamboni, \emph{Periodicity and local complexity}, Theoret.
  Comput. Sci. \textbf{319} (2004), 229--240.

\bibitem[Shi17]{Shimomura:2017}
T.~Shimomura, \emph{Finite-rank {B}ratteli-{V}ershik homeomorphisms are
  expansive}, Proc. Amer. Math. Soc. \textbf{145} (2017), 4353--4362.

\end{thebibliography}
\bibliographystyle{amsalpha}

\end{document}